\documentclass{article}
\usepackage[utf8]{inputenc}
\usepackage[left=1in,top=1in,right=1in,bottom=1in]{geometry}
\usepackage[linktocpage=true]{hyperref}
\usepackage{setspace}
\usepackage{amssymb, amsmath, amsthm, graphicx,mathrsfs}
\usepackage{caption,cite}
\usepackage{mathtools,color}

\usepackage{cleveref}
\usepackage{soul}

\newtheorem{thm}{Theorem}[section]
\newtheorem{cor}[thm]{Corollary}

\newtheorem{defn}{Definition}

\newtheorem{lem}[thm]{Lemma}

\newtheorem{prop}[thm]{Proposition}
\newtheorem{claim}[thm]{Claim}

\newtheorem{quest}{Question}
\newtheorem{prob}{Problem}

\newcommand{\R}{\mathbb{R}}

\renewcommand{\t}{\text}

\renewcommand{\l}{\left}
\renewcommand{\r}{\right}
\newcommand{\ex}{\mathrm{ex}}

\renewcommand{\SS}[1]{\textcolor{red}{#1}}

\renewcommand{\c}[1]{\mathcal{#1}}
\newcommand{\sub}{\subseteq}
\newcommand{\Om}{\Omega}
\newcommand{\sm}{\setminus}
\newcommand{\f}[2]{\frac{#1}{#2}}
\newcommand{\rec}[1]{\frac{1}{#1}}

\newcommand{\sig}{\sigma}

\newcommand{\Del}{\Delta}
\newcommand{\del}{\delta}
\newcommand{\half}{\frac{1}{2}}
\newcommand{\floor}[1]{\lfloor #1\rfloor}

\newcommand{\proj}{\mathrm{proj}}

\title{Random Tur\'an Problems for $K_{s,t}$ Expansions}

\author{Jiaxi Nie\footnote{School of Mathematics, Georgia Institute of Technology {\tt jnie47@gatech.edu}.}\and Sam Spiro\footnote{Dept.\ of Mathematics, Rutgers University {\tt sas703@scarletmail.rutgers.edu}. This material is based upon work supported by the National Science Foundation Mathematical Sciences Postdoctoral Research Fellowship under Grant No. DMS-2202730.}}
\date{\today}

\begin{document}
	\maketitle
	
	\begin{abstract}
		Let $K_{s,t}^{(r)}$ denote the $r$-uniform hypergraph obtained from the graph $K_{s,t}$ by inserting $r-2$ new vertices inside each edge of $K_{s,t}$.  We prove essentially tight bounds on the size of a largest $K_{s,t}^{(r)}$-subgraph of the random $r$-uniform hypergraph $G_{n,p}^r$ whenever $r\ge 2s/3+2$, giving the first random Tur\'an results for expansions that go beyond a natural ``tight-tree barrier.''  In addition to this, our methods yield optimal supersaturation results for $K_{s,t}^{(3)}$ for sufficiently dense host hypergraphs, which may be of independent interest.
	\end{abstract}
	
	\section{Introduction}
	Given an $r$-uniform hypergraph $F$ (or $r$-graph for short), we define the \textit{Tur\'an number} $\ex(n,F)$ to be the maximum number of edges that an $n$-vertex $F$-free $r$-graph can have.  Let $G_{n,p}^r$ denote the random $r$-graph on $n$ vertices obtained by including each possible edge independently and with probability $p$, and when $r=2$ we simply write $G_{n,p}$ instead of $G_{n,p}^2$.  We define the \textit{random Tur\'an number} $\ex(G_{n,p}^r,F)$ to be the maximum number of edges in an $F$-free subgraph of $G_{n,p}^r$.  Note that when $p=1$ we have $\ex(G_{n,1}^r,F)=\ex(n,F)$, so the random Tur\'an number can be viewed as a probabilistic analog of the classical Tur\'an number.
	
	The asymptotics for $\ex(G_{n,p}^r,F)$ are essentially known if $F$ is not an $r$-partite $r$-graph due to independent breakthrough work of Conlon and Gowers \cite{conlon2016combinatorial} and of Schacht \cite{schacht2016extremal}, and because of this we focus only on the degenerate case when $F$ is $r$-partite.  This degenerate case seems to be very difficult even in the graph setting $r=2$ where effective upper bounds are known only  for even cycles \cite{morris2016number,jiang2024balanced}, complete bipartite graphs~\cite{morris2016number}, and theta graphs~\cite{mckinley2023random}; with effective lower bounds known only for powers of rooted trees~\cite{spiro2024random}.  All of these bounds agree with a recent conjecture of McKinely and Spiro~\cite{mckinley2023random}, which gives a general prediction for how the function $\ex(G_{n,p},F)$ should behave for all bipartite $F$.
	
	Much less is known for $r$-graphs with $r\ge 3$, which is partially due to a surprising connection between the random Tur\'an problem and Sidorenko's conjecture~\cite{nie2023sidorenko}.  The few examples of hypergraphs that we do understand tend to be hypergraphs which ``look like'' graphs which we already know how to solve the random Tur\'an problem for, such as for  complete $r$-partite $r$-graphs~\cite{spiro2022counting} and loose even cycles~\cite{jiang2024number,mubayi2023random,nie2024turan}.  One broad class of $r$-graphs in this direction for which recent success has been found is the following.
	\begin{defn}
		Given an $r_0$-graph $F$, we define its \textit{$r$-expansion} $F^{(r)}$ to be the $r$-graph obtained by inserting $r-r_0$ distinct new vertices into each edge of $F$.
	\end{defn}
	Hypergraph expansions are a natural and well-studied class of hypergraphs, especially in the context of (classical) Tur\'an numbers.  For example, the famed Erd\H{o}s matching conjecture~\cite{erdos1965problem} is equivalent to determining $\ex(n,M_k^{(r)})$ when $M_k$ is a graph matching of size $k$.  There are many more results for Tur\'an numbers of expansions \cite{chung1983unavoidable,kostochka2015turan,mubayi2007intersection,pikhurko2013exact}, and we refer the interested reader to the survey by Mubayi and Verstra\"ete~\cite{mubayi2016survey} for much more on (classical) Tur\'an problems for expansions.
	
	In \cite{nie2024random} we initiated the systematic study of random Tur\'an numbers of expansions by proving a set of ``lifting'' results which informally says that one can often  take bounds for the random Tur\'an problem of an $r_0$-graph $F$ and ``lift'' these to bounds for the random Tur\'an problem of  expansions $F^{(r)}$. These results  build on and generalize earlier works of Mubayi and Yepremyan~\cite{mubayi2023random} and Nie~\cite{nie2024turan}, culminating in the following result which solves the random Tur\'an problem for large expansions of strictly balanced Sidorenko hypergraphs\footnote{We omit defining what it means for a hypergraph to be strictly balanced or Sidorenko, as their only relevance for the present paper is the fact that complete bipartite graphs are both strictly balanced and Sidorenko.}.  For this, we say that a sequence of events $A_n$ holds \textit{asymptotically almost surely} or \textit{a.a.s.}\ for short if $\Pr[A_n]\to 1$ as $n\to \infty$, and we write $f(n)\ll g(n)$ if $f(n)/g(n)\to 0$ as $n\to \infty$. 
	
	\begin{thm}[\cite{nie2024random}]\label{thm:strongExpansion}
		Let $F$ be a strictly $r_0$-balanced $r_0$-graph with $2\le \Del(F)<|F|$.  If $F$ is Sidorenko, then for all $r> |F|^2v(F)r_0$, we have a.a.s.
		$$
		\ex(G^r_{n,p}, F^{(r)})=
		\l\{
		\begin{aligned}
			& \max\l\{\Theta(pn^{r-1}),n^{r_0-\frac{v(F)-r_0}{|F|-1}}(\log n)^{\Theta(1)}\r\},~~&\t{if}~p\gg n^{-r+r_0-\frac{v(F)-r_0}{|F|-1}}\\
			&(1+o(1))p\binom{n}{r},~~&\t{if}~ n^{-r}\ll p\ll  n^{-r+r_0-\frac{v(F)-r_0}{|F|-1}}.\\
		\end{aligned}
		\r. $$
		Moreover, if there does not exist any $k$ such that $F^{(k)}$ is Sidorenko, then these bounds fail to hold for all $r$.
	\end{thm}
	
	We proved \Cref{thm:strongExpansion} by ``lifting'' a general random Tur\'an bound for Sidorenko hypergraphs due to Jiang and Longbrake~\cite{jiang2024balanced}.  By lifting more specialized results, one can often obtain better dependencies on $r$ than what is given in \Cref{thm:strongExpansion} for specific choices of $F$.  In particular, applying our general lifting results together with the known (tight) bounds for the random Tur\'an number of complete bipartite graphs due to Morris and Saxton~\cite{morris2016number} gives the following, establishing tight bounds for $K_{s,t}^{(r)}$ whenever $r\ge s+2$.

	\begin{thm}[\cite{nie2024random}]\label{thm:OldKst}
		If $r,s,t\ge 2$ are integers such that $t\ge s$ and either
		\begin{itemize}
			\item $r\ge s+2$, or
			\item $r\ge s+1$ and $t\ge s^2-2s+3$,
		\end{itemize}
		
		then for all $0<p=p(n)\le 1$ we have a.a.s.
		$$
		\ex(G^r_{n,p}, K_{s,t}^{(r)})=
		\l\{
		\begin{aligned}
			& \max\{\Theta(pn^{r-1}),n^{2-\frac{s+t-2}{st-1}}(\log n)^{\Theta(1)}\},~~&\t{if}~p\gg n^{-r+2-\frac{s+t-2}{st-1}}\\
			&(1+o(1))p\binom{n}{r},~~&\t{if}~ n^{-r}\ll p\ll  n^{-r+2-\frac{s+t-2}{st-1}}.\\
		\end{aligned}
		\r.
		$$
	\end{thm}
	The difference between $r\ge s+2$ and $r\ge s+1$ here is of course marginal for large $s$, but in the case $s=2$, this stronger $s+1$ bound is needed to give a positive answer to the following natural followup question to \Cref{thm:strongExpansion}.
	
	\begin{quest}\label{prob:Sidorenko}
		Which Sidorenko $r_0$-graphs $F$ are such that the tight bounds of \Cref{thm:strongExpansion} hold for every $r$ with $\ex(n,F^{(r)})=\Theta(n^{r-1})$?
	\end{quest}
	
	We note that the hypothesis $\ex(n,F^{(r)})=\Theta(n^{r-1})$ is necessary for the bounds of \Cref{thm:strongExpansion} to hold, as otherwise the stated bound for $p=1$ would be incorrect.  
	
	\Cref{thm:strongExpansion} shows that large expansions of Sidorenko hypergraphs trivially give a positive answer to \Cref{prob:Sidorenko}.  The only known non-trivial examples of positive answers towards \Cref{prob:Sidorenko} are certain classes of hypergraph trees~\cite[Theorem 1.5]{nie2023random}, as well as the graphs $K_{2,t}$ with $t\ge 3$ 
	(as exhibited by \Cref{thm:OldKst}) and for theta graphs with sufficiently many paths \cite{nie2024random}, with no negative results currently being known for this question.  
	
	We focus on \Cref{prob:Sidorenko} for complete bipartite graphs, and more generally on trying to improve the dependency on $r$ from \Cref{thm:OldKst} in this case.  Towards this end, we note that known methods can be used to show $\ex(n,K_{s,t}^{(r)})=\Theta(n^{r-1})$ whenever ${r\choose 2}\ge s$ \cite{kostochka2015turan} with these bounds no longer holding for smaller $r$ and sufficiently large $t$ \cite[Theorem 1.4]{ma2018some}.  As such, we would ideally like to extend the range $r\ge s+1$ from \Cref{thm:OldKst} to something like $r\ge \sqrt{2s}+1$.  
	
	As we will discuss in greater depth in \Cref{sub:sketch} below, there is a real technical  barrier to proving tight bounds for $\ex(G_{n,p}^r,K_{s,t}^{(r)})$ when $r\le s$ (let alone for $r\approx \sqrt{2s}$), as in the range $r\le s$ one can no longer use the tight-tree approach initiated by the first author~\cite{nie2023random} which builds upon ideas from Balogh, Narayanan and Skokan~\cite{balogh2019number}.
	
	The main result of this paper is to go beyond the troublesome tight-tree barrier mentioned above and extend \Cref{thm:OldKst} from roughly $r\ge s$ to $r\ge 2s/3$ as follows
	\begin{thm}\label{thm:Kst}
		If $r,s,t\ge 3$ are integers such that $t\ge s$ and either
		\begin{itemize}
			\item $r\ge 2s/3+2$, or
			\item $r\ge \lfloor 2s/3+2\rfloor$ and $t\ge 8s/3$,
		\end{itemize}
		then for all $0<p=p(n)\le 1$ we have a.a.s.
		$$
		\ex(G^r_{n,p}, K_{s,t}^{(r)})=
		\l\{
		\begin{aligned}
			& \max\l\{pn^{r-1},n^{2-\frac{s+t-2}{st-1}}\r\}(\log n)^{\Theta(1)},~~&\t{if}~p\gg n^{-r+2-\frac{s+t-2}{st-1}}\\
			&(1+o(1))p\binom{n}{r},~~&\t{if}~ n^{-r}\ll p\ll  n^{-r+2-\frac{s+t-2}{st-1}}.\\
		\end{aligned}
		\r.
		$$
	\end{thm}
	Note that in the case $s=4$, this result gives tight random Tur\'an bounds for $K_{4,t}^{(r)}$ whenever $r\ge 4$ and $t$ is sufficiently large.  Since this is exactly the range in which ${r\choose 2}\ge s$, we conclude the following.
	\begin{cor}\label{cor:4}
		The graph $K_{4,t}$ satisfies \Cref{prob:Sidorenko} for all $t$ sufficiently large.
	\end{cor}
	The reader may be surprised at this point to see that the only $K_{s,t}$ for which positive answers are known to \Cref{prob:Sidorenko} are for $s=2$ and $s=4$ (notably with us not able to do this for the case $s=3$ and $r=3$).  This peculiar situation occurs because giving a positive answer to \Cref{prob:Sidorenko} for an integer $s$ of the form $1+{k-1\choose 2}$ is essentially easier than giving a positive answer for any other $s'$ lying between $1+{k-1\choose 2}$ and ${k\choose 2}$.  Indeed, for all of these values of $s'$, we primarily need to show $\ex(G_{n,p}^{(r)},K_{s',t}^{(r)})=O(n^{r-1})$ for all $r\ge k$, and this is easiest to do for the case $s=1+{k-1\choose 2}$ since $K_{s,t}^{(r)}\sub K_{s',t}^{(r)}$ for all $s'\ge s$.

	Our proof of \Cref{thm:Kst} relies on a technical balanced supersaturation result \Cref{prop:technicalSupersaturation}, which as a corollary yields the following supersaturation result of independent interest.  For this, we use the standard notation $f=\Om^*(g)$ to mean that there exists an absolute constant $C$ so that $f(n)\ge (\log n)^{-C} g(n)$.  
	
	\begin{prop}\label{prop:vanillaSupersaturation}
		For all $t\ge s\ge 3$, there exists a sufficiently large constant $C$ such that if $H$ is an $n$-vertex $3$-graph with $m$ edges such that $m\ge n^{3-3/s+1/st}(\log n)^{C}$, then $H$ contains at least
		$$
		\Omega^*(m^{st}n^{s+t-2st})
		$$
		copies of $K^{(3)}_{s,t}$.
	\end{prop}
	The count of $\Omega^*(m^{st}n^{s+t-2st})$ is tight up to the implicit logarithmic factors, as can be seen by considering $H$ to be a random 3-graph with $m$ hyperedges, and it is possible that this tight bound continues to hold under the weaker hypothesis of $m=\Omega(n^{3-3/s})$.  We discuss this possible strengthening and other extensions of \Cref{prop:vanillaSupersaturation} in the concluding remarks of the paper.

	\subsection{Proof Outline}\label{sub:sketch}
	Here we very briefly sketch the high-level ideas of our proof for \Cref{thm:Kst}. The lower bounds for the theorem follow easily from known results in the literature, so all that is needed is to prove the corresponding upper bounds.  
	
	As is often the case with random Tur\'an problems, we will prove our upper bounds by proving a ``balanced supersaturation result'', which informally says that that if an $r$-graph $H$ has many edges, then we can find a collection $\c{H}$ of copies of $K_{s,t}^{(r)}$ in $H$ which are ``spread out'' in the sense that no set of edges of $H$ lies in too many copies of $K_{s,t}^{(r)}$ in $\c{H}$.  
	
	In order to further discuss the details of our proof, we define for an $r$-graph $H$ its \textit{$k$-shadow} $\partial^k H$ to be the $k$-graph on $V(H)$ where a given $k$-set $K$ is an edge of $\partial^k H$ if and only if  $K$ is contained in an edge of $H$.
	
	\textbf{Balanced Supersaturation for $r\ge s+1$}.  At a very high-level, our previous proof of \Cref{thm:OldKst} came from a balanced supersaturation result which itself was proven by first (I) finding a (large) collection of graph $K_{s,t}$'s in the shadow $\partial^2 H$, and then (II) extending these copies of $K_{s,t}$ into copies of $K_{s,t}^{(r)}$ in $H$.  Step (II) here is rather straightforward to implement, so the main difficulty lies in establishing (I). 
	
	To prove (I), i.e.\ to find a large collection of graph $K_{s,t}$'s in the shadow $\partial^2 H$, we roughly speaking build our collection in one of two ways depending on the structure of $\partial^2 H$.  First, if $\partial^2 H$ has many edges, then a known balanced supersaturation result of Morris and Saxton~\cite{morris2016number} for $K_{s,t}$ immediately implies that $\partial^2 H$ contains many copies of $K_{s,t}$.  Otherwise if $\partial^2 H$ is sparse, then one can argue that there are many $s$-sets $\{u_1,\ldots,u_s\}\sub V(H)$ which are contained in many edges of $H$; and it is in this step here that we crucially need $r\ge s+1$ (as otherwise an $s$-set will be contained in at most one edge of $H$).  By using these many edges of $H$ containing $\{u_1,\ldots,u_s\}$, one can show that there exist many vertices $v$ such that $\{u_1,\ldots,u_s,v\}$ is contained in an edge of $H$, and taking $t$ such sets implies the existence of many $K_{s,t}$'s in $\partial^2 H$. %(with the additional property that the part of size $s$ forms a clique).
	
	Slightly more generally, it is implicitly shown in \cite[Theorem 2.8, Lemma 4.3]{nie2024random} that for any graph $F$, there exists an integer $t(F)$ such that for $r\ge t(F)$, an analog of (I) can be proven for $F$ through an argument similar to the one outlined above.  At a very high level, the parameter $t(F)$ is the smallest quantity such that it is ``easy'' to prove $\ex(n,F^{(r)})=O(n^{r-1})$ for $r\ge t(F)$, and intuitively this implies that it should also be ``easy'' to prove a corresponding bound for the random Tur\'an problem in this same range of $r$.  More precisely, $t(F)$ is defined to be the smallest $t$ such that $F^{(t)}$ is a spanning subgraph of a ``tight tree'' of uniformity $t$.  
	
	Prior to the present work, tight bounds for the random Tur\'an problem for expansions were known only in the range $r\ge t(F)$, since below this point several of the key technical results in the area fail to go through.  As such, we refer to $t(F)$ as the ``tight-tree barrier'' for random Tur\'an problems of hypergraph expansions, with the main objective of this paper being the development of ideas to go beyond this barrier in the case of complete bipartite graphs (and possibly beyond).
	
	\textbf{Balanced Supersaturation for Smaller $r$}.  Our key innovation for proving our stronger bound of $r\ge 2s/3+2$ in \Cref{thm:Kst} is that we no longer build our large collection of copies of $K_{s,t}^{(r)}$ in $H$ from copies of $K_{s,t}$ in the \textit{2-shadow }$\partial^2 H$, but rather from copies of $K_{s,t}^{(3)}$ in the \textit{3-shadow }$\partial^3 H$.  Specifically, we will prove a balanced supersaturation result for $K_{s,t}^{(3)}$ which (unlike the result of Morris and Saxton for the 2-shadow case) will be strong enough to be applied regardless of the density of $\partial^3 H$, allowing us to build a large collection of copies of $K_{s,t}^{(3)}$ without needing to rely on the existence of sets $\{u_1,\ldots,u_s\}$ contained in many edges, letting us go below the $r\ge s+1$ barrier that was intrinsic to previous approaches.
	
	It remains then for us to prove this balanced supersaturation result for $K_{s,t}^{(3)}$.  Our proof will rely on two balanced supersaturation results for graphs: one which is effective when $\partial^2H$ is dense and one which is effective for sparse $\partial^2 H$. As before, the case when $\partial^2 H$ has many edges will use the result of Morris and Saxton~\cite{morris2016number} mentioned above.  For the second (and far more technical) case of $\partial^2H$ being sparse, we develop a new balanced supersaturation result for $K_{s,t}$ which is effective for tripartite graphs which contain many triangles but few edges, the proof of which is based implicitly on the idea of \textit{vertex supersaturation} as developed in \cite{mckinley2023random}.  Combining these two approaches gives our desired balanced supersaturation result for $K_{s,t}^{(3)}$, from which we can conclude our bounds in \Cref{thm:Kst} through the usage of our general theorems for expansions developed in \cite{nie2024random}.

	\section{Preliminaries}
	In this section, we gather together the three main technical lemmas of this paper (balanced supersaturation when $\partial^2 H$ is dense, balanced supersaturation when $\partial^2 H$ is sparse, and a lemma translating balanced supersaturation for $K_{s,t}$  into balanced supersaturation for $K_{s,t}^{(3)}$), together with two extra auxiliary results.   We will ultimately also need to work with some of the general tools from \cite{nie2024random} regarding random Tur\'an problems for expansions, but we postpone stating these results until they are needed in \Cref{sub:finish} for ease of reading.
	
	\subsection{Definitions and Conventions}
	
	In order to distinguish between the various types of objects used throughout the paper, we will work with the following set of notational conventions: graphs will be denoted by $G$, 3-graphs will be denoted by $H$, collections of copies of $K_{s,t}$ in a graph $G$ will be denoted by $\c{G}$, and collections of copies of $K_{s,t}^{(3)}$ in a 3-graph $H$ will be denoted by $\c{H}$.  We will treat $\c{G}$ and $\c{H}$ as $st$-uniform hypergraphs on $E(G)$ and $E(H)$, respectively.  For clarity, we refer to elements of $\c{G},\c{H}$ as hyperedges and elements of $G,H$ as edges.  To further distinguish between elements of $G,H$, we will let $e\in G$ denote an edge of $G$ and $\sig'\sub G$ denote a set of edges of $G$, while in contrast we let $h\in H$ denote an edge of $H$ and $\sig\sub H$ denote sets of edges of $H$.
	
	We recall that we use the notation $f=\Om^*(g)$ to mean that there exists an absolute constant $C$ so that $f(n)\ge C^{-1}(\log n)^{-C} g(n)$, and that $\partial^2 H$ denotes the graph on $V(H)$ whose edges are pairs of $H$ that lie in an edge of $H$.  We will also make use of the following non-standard definition which we highlight below for ease of recall.
	\begin{defn}
		Given a set of edges $\sig'\sub E(G)$ of a graph $G$, we define $\sig'_v\sub V(G)$ to be the set of vertices belonging to at least one edge of $\sig'$.
	\end{defn}

	To help guide the reader through some of the more technical aspects of our proof, we have included several asides discussing the intuition behind the various ideas that we use.  To help distinguish these asides from the full formal details of the argument, we separate these two portions of the discussion with labels \textbf{Intuition} and \textbf{Formal Details}.
	
	\subsection{Dense Balanced Supersaturation}
	The following result was essentially proven by Morris and Saxton.
	\begin{prop}[Theorem 7.4, \cite{morris2016number}] \label{prop:dense}
		For all $s,t\ge 2$ there exists an $L_0>0$ such that the following holds.  Let $G$ be a bipartite graph with bipartition $U\cup V$ which has $|U|\le |V|$ and $L |V|^{2-1/s}$ edges.  If $L\ge L_0$, then there exists a collection $\c{G}$ of $K_{s,t}$'s in $G$ whose parts of size $s$ lie in $U$ such that $|\c{G}|=\Omega(L^{st} |V|^s)$, and such that every $\sig'\sub E(G)$ with $|\sig_v'\cap U|=a\ge 1$ and $|\sig_v'\cap V|=b\ge 1$ satisfies \[\deg_{\c{G}}(\sig')\le O((L |V|^{1-1/s})^{s-a}(L^s)^{t-b}).\]
	\end{prop}
	%\SS{Where do we use the fact that we can force the part of size $s$ onto the smaller side? I think really here I just wanted to ensure that all of the $K_{s,t}$'s went to the same side, and it would be convenient to predict which part it is ahead of time but if we wanted to we could ignore this and just do two cases based on which one is the majority.  In fact, maybe we could just use the version of the codegrees where we justlook at $|\sig|$ instead of $a,b$ to get the same result.}
	
	Strictly speaking, Morris and Saxton prove a slightly different result than that of \Cref{prop:dense} (which would also suffice for our purposes at the cost of complicating our notation), but their proof easily adapts to prove the result above.  More precisely, there are two differences between the  statement of Morris and Saxton's  with \Cref{prop:dense}.  First, they do not consider the setting of bipartite graphs, but one can easily reduce to this case by considering a bipartite subgraph of $G$ containing at least half its edges.  In the bipartite setting, their result as written does not say whether the copies of $K_{s,t}$ have their parts of size $s$ in $U$ or in $V$, but it is easy to see from their proof (which is essentially a refined version of the classic proof of the K\H{o}v\'ari-S\'os-Tur\'an theorem) that we can guarantee that every $K_{s,t}$ in $\c{G}$ has their set of size $s$ in the smaller set $U$, from which our statement follows.
	\subsection{Sparse Balanced Supersaturation}
	\textbf{Intuition}.  The sparse case of our argument will require a more complicated analog of \Cref{prop:dense} for tripartite graphs.  We note that this result will not be needed until halfway through the proof of \Cref{prop:technicalSupersaturation}, and as such, the reader may wish to postpone fully digesting the statement of the result until this point.
	
	The basic idea of our sparse balanced supersaturation result is roughly the same as that of \Cref{prop:dense}: we want to say that if $G$ is a ``nice'' tripartite graph (i.e. if $G$ is close to regular, has many triangles, and has few edges), then we can find a collection $\c{G}$ of $K_{s,t}$'s in $G$ such that (1) we can specify how each part of size $s$ in the $K_{s,t}$'s intersects the  parts of $G$, (2) $|\c{G}|$ is large, and (3) we can bound $\deg_{\c{G}}(\sig')$ by some function $\phi$ depending only on how the vertices $\sig'_v$ intersect the parts of $G$.  For technical reasons, we will also need to further impose that $G$ is the shadow of a ``nice'' hypergraph $H$.  
	
	\textbf{Formal Details}.  We begin by specifying the function $\phi$ mentioned above.
	
	\begin{defn}
		Let $G$ be a tripartite graph on $V_1\cup V_2\cup V_3$ and  $n,\del,K,\ell,m_{1,3},m_{2,3}$ a set of given parameters.  We define a function $\phi:2^{V(G)}\to \R\cup\{\infty\}$ as follows: given $\nu\sub V(G)$, let $a=|\nu\cap (V_1\cup V_2)|$ and $b=|\nu\cap V_3|$.  If $a=0$ or $b=0$, then $\phi(\nu)=\infty$.  Otherwise, if $\nu\cap V_1\ne \emptyset$ we define 
		\[\phi(\nu)=\frac{\ell^{s-1} n^{s-1+1/s}(\ell^{3-2s}K^{s-1})^t}{\del m_{1,3}(\del \ell^{-1} K n^{1-2/s})^{a-1}(\del \ell^{3-2s} K^{s-1})^{b-1}}\]
		and if $\nu\cap V_1=\emptyset $ we define
		\[\phi(\nu)=\frac{\ell^{s-1} n^{s-1+1/s}(\ell^{3-2s}K^{s-1})^t}{\del m_{2,3}(\del \ell^{-1} K n^{1-2/s})^{a-1}(\del \ell^{3-2s} K^{s-1})^{b-1}}.\]
	\end{defn}
	
	We emphasize that the only difference between the two inlined expression is the usage of $m_{1,3}$ and $m_{2,3}$ in the denominator.   We also note that having $\phi(\nu)=\infty$ when either $a=0$ or $b=0$ is merely a matter of convenience, as we will never really work with such $\nu$.  With this we can now state our balanced supersaturation result for ``nice'' tripartite graphs (with the intuition for its various hypothesis being given directly after the statement of the proposition).
	
	\begin{prop}\label{prop:sparse}
		For all $s,t\ge 3$, there exists a sufficiently small $\del>0$  such that the following holds.  Let $\ell,K,n$ be real numbers and let $H$ be a tripartite 3-graph with tripartition $V_1\cup V_2\cup V_3$ with $m_{i,j}$  the number of pairs in $V_i\cup V_j$ contained in at least one hyperedge.  Further, assume $\ell,K,n, H$ satisfy the following:
		\begin{itemize}
			\item[(a)] $|V_i|\le n$ for all $i$.
			\item[(b)] $H$ has at least $K n^{3-3/s}(12\log n)^{16}$ hyperedges.
			\item[(c)] $m_{i,j}\le \ell |V_1|^{2-1/s}$ for all $i,j$.
			\item[(d)] We have $\del^{-1}\le \ell\le \del K^{(s-1)/(2s-3)}$. 
			\item[(e)] Each pair in $V_i\cup V_j$ is contained in at most $K n^{3-3/s} m_{i,j}^{-1} (12\log n)^{32}$ hyperedges for all $i,j$.
			\item[(f)] Each vertex in $V_1$ has at most $\ell |V_1|^{1-1/s} (12\log n)^{16}$ neighbors in $V_2$.
		\end{itemize}
		Then for $G=\partial^2 H$, there exists a collection $\c{G}$ of $K_{s,t}$'s in $G$ such that every $K_{s,t}$ intersects $V_1$ in 1 vertex and $V_2$ in $s-1$ vertices, such that $|\c{G}|=\Om^*(\ell^{s-1} n^{s-1+1/s}(\ell^{3-2s}K^{s-1})^t)$, and such that for any $\sig'\sub E(G)$ we have $\deg_{\c{G}}(\sig')\le \phi(\sig'_v)$.
	\end{prop}
	%We note that the only reason we introduce $H$ instead of working with $G$ directly is because when passing from $H$ to $G$, we can no longer guarantee that each edge in $V_i\cup V_j$ is contained in few triangles (since there can exist triangles in $G$ which are not hyperedges in $H$).
	
	While the hypothesis in \Cref{prop:sparse} look intimidating at first glance, each individual part is fairly straightforward to understand on its own.  For example, (a) essentially says that $H$ is an $n$-vertex hypergraph, (b) says that $H$ has many edges, and (c), (d) combine to say that $H$ has a relatively sparse shadow graph. Conditions (e) and (f) are ``regularity'' conditions, with (e) saying none of the $m_{i,j}$ pairs in the shadow of $V_i\cup V_j$ are contained in substantially more than the average number of edges of $H$, and (f) says that none of the $V_1$ vertices have substantially more than the average degree into $V_2$. For ease of reading, we will postpone the proof of \Cref{prop:sparse} until \Cref{subsec:Sparse}.
	
	\subsection{Translating from Graphs to Hypergraphs}
	
	\textbf{Intuition}. We will use the two results above to find a large  ``spread out'' collection $\c{G}$ of graph $K_{s,t}$'s in the shadow graph of a hypergraph $H$.  Given such a collection $\c{G}$, we will form a collection $\c{H}$ of $K_{s,t}^{(3)}$'s in $H$ by  selecting copies $F'\in \c{G}$ and then iteratively extending each edge $e\in F'$ into an edge $h\in H$ containing it. 
	In doing this, we will want to make sure our collection of $K_{s,t}^{(3)}$'s is still ``spread out'', i.e.\ that $\deg_{\c{H}}(\sig)$ is not too large for any given set of hyperedges $\sig$.  
	
	In order to bound $\deg_{\c{H}}(\sig)$ for some given $\sig\sub H$, we will need to consider all the possible ways that each of the edges $h\in \sig$ could have been generated by this procedure.  That is, we will need to consider all the possible ways of picking for each edge $h\in \sig$ an edge $e_h\in G$ contained in $h$ (which we think of as the edge in a copy $F'$ that we extended to form $h$), and then work with the set $\sig'=\{e_h:h\in \sig\}$.  The set of all such $\sig'$ that can be formed in this way will be denoted by $\proj(\sig)$.  More precisely, we use the following definition.

	\textbf{Formal Details}.
	
	\begin{defn}
		If $\sig$ is a collection of 3-sets such that $|h\cap h'|\le 1$ for all distinct $h,h'\in \sig$, then we define the \textit{projection} $\proj(\sig)$ to be the collection of all sets of 2-sets $\sig'$ obtained by including exactly one pair from each $h\in \sig$. 
	\end{defn}
	For example, if $\sig=\{123,345\}$ then
	\[\proj(\sig)=\{\{12,34\},\{12,35\},\{12,45\},\{13,34\},\{13,35\},\{13,45\},\{23,34\},\{23,35\},\{23,45\}\}.\]

	We now state our main translation lemma from copies of $K_{s,t}$ into $K_{s,t}^{(3)}$.  We in fact state this in the more general context of arbitrary graphs $F$ since proof and statement of the result are no more difficult in this more general setting. 
	
	\begin{lem}\label{lem:translate}
		Let $F$ be a graph, let $H$ be a $3$-partite $3$-graph with partition $V_1\cup V_2\cup V_3$, let $G_{i,j}$ denote the pairs of vertices in $V_i\cup V_j$ contained in at least one edge of $H$, and let $\c{G}$ be a collection of copies of $F$ in the shadow graph of $H$.  If there exist real numbers $D_{i,j},d_{i,j}\ge 2v(F)+2|F|$ such that $d_{i,j}\le \deg_{H}(e)\le D_{i,j}$ for every pair $e\in G_{i,j}$, and if there exists integers $c_{i,j}$ such that every copy of $F\in \c{G}$ has exactly $c_{i,j}$ edges in $G_{i,j}$, then there exists a collection $\c{H}$ of expansions $F^{(3)}$ in $H$ such that \[|\c{H}|=\Om\left(|\c{G}|\prod_{i,j} d_{i,j}^{c_{i,j}}\right)\] and such that for any $\sig\sub H$, we have $\deg_{\c{H}}(\sig)=0$ if $|h\cap h'|>1$ for some distinct $h,h'\in \sig$, and otherwise.
		\[\deg_{\c{H}}(\sig)\le \sum_{\sig'\in \proj(\sig)} \deg_{\c{G}}(\sig')\prod_{i,j} D_{i,j}^{c_{i,j}-|\sig'\cap G_{i,j}|}.\]
	\end{lem}
	
	Variants of this translation lemma have appeared in essentially every previous work on the random Tur\'an problem for expansions, such as \cite[Lemma 2.1]{nie2024random}.  We note, however, that all of these previous works  only ever considered the case that each copy of $F$ lied entirely within two of the parts $V_i,V_j$ (so that $c_{i,j}=|F|$ and $c_{i',j'}=0$ otherwise), allowing for a much simpler statement of the result.  For our purposes, however, it will be crucial that we allow our copies of $F$ to lie in all three parts of $H$.
	\begin{proof}
		Consider the following procedure for constructing copies of $F^{(3)}$ in $H$. Start by picking some $F'\in \c{G}$ and let $e_1,\ldots,e_m$ be an arbitrary ordering of the edges of $F'$.  Iteratively pick $h_k\in E(H)$ to be any edge of $H$ which contains $e_k$ and which is disjoint from $(V(F')\cup \bigcup_{k'<k} h_{k'})\setminus e_k$.  Observe that if this process terminates, then these $h_k$ edges form a copy of $F^{(3)}$ in $H$, and we let $\c{H}$ denote the set of copies of $F^{(3)}$ that can be formed in this way.  It remains to analyze the properties of this collection.
		
		Observe that if $e_k\in G_{i,j}$, then the number of choices for $h_k$ is at least $d_{i,j}-v(F)-|F|\ge \half d_{i,j}$, so the total the number of ways of completing this algorithm is at least $\Om(|\c{G}|\prod_{i,j} d_{i,j}^{c_{i,j}})$.   Moreover, since each $F^{(3)}\in \c{H}$ arises in at most $O(1)$ ways form this procedure, we conclude that $|\c{H}|$ has the desired size. 
		
		Now consider any $\sig\sub H$.  If $|h\cap h'|>1$ for some distinct $h,h'\in \sig$, then $\deg_{\c{H}}(\sig)=0$ (since $\sig$ can not be the subset of any expansion $F^{(3)}$), so we can assume this is not the case.  Observe that every copy of $F^{(3)}\in \c{H}$ containing $\sig$ arose 
		by starting with some $F'\in \c{G}$ which contained a graph edge from each $h\in \sig$ (i.e.\ which contained $\sig'\sub F'$ for some $\sig'\in \proj(\sig)$), then extending each of these graph edges to form $\sig$, and then extending the remaining graph edges $F'\sm \sig'$ in an arbitrary way.  Thus one can specify each copy of $F^{(3)}$ containing $\sig$ through the following steps:
		\begin{enumerate}
			\item[(a)] Choose some $\sig'\in \proj(\sig)$.
			\item[(b)] Choose some $F'\in \c{G}$ containing $\sig'$ (which can be done in at most $\deg_{\c{G}}(\sig')$ ways).
			\item[(c)] Choose extensions of each edge in $F\sm \sig'$ (which can be done in at most $\prod_{i,j} D_{i,j}^{c_{i,j}-|\sig'\cap G_{i,j}|}$ ways).
		\end{enumerate}
		Multiplying these number of choices  from (b) and (c) together and then summing over all possible choices of $\sig'$ in (a) gives the desired degree bound, finishing the proof.
	\end{proof}
	
	\subsection{Auxiliary Lemmas}
	
	We will make use of the following basic arithmetic lemma.
	\begin{lem}\label{lem:optimize}
		Let $\pi,A,B$ be real numbers and $s,t\ge 1$ integers satisfying $A\ge B$ and $\pi\ge B^{-1},(A^{s-1}B^{t-1})^{\frac{-1}{st-1}}$.  Then 
		\[\max_{1\le a\le s,\ 1\le b\le t}\pi^{1-ab} A^{1-a}B^{1-b}\le 1.\]
	\end{lem}
	\begin{proof}
		Fix some $a,b$ achieving this maximum, and observe that the quantity we wish to upper bound can be rewritten as \[(\pi A)^{1-a} (\pi^a B)^{1-b}.\]  Since $\pi A\ge 1$ (because $\pi\ge B^{-1}\ge A^{-1}$), and since $a,b\ge 1$, we see that this quantity will be at most 1 unless $\pi^a B<1$.  We may thus assume this holds, in which case the expression above is maximized when $b\le t$ is as large as possible, giving an upper bound of
		\[(\pi A)^{1-a} (\pi^a B)^{1-t}=B^{1-t}\pi A(\pi^t A)^{-a}.\]
		Observe that this quantity is maximized at either $a=1$ or $a=s$.  In the former case we end up with $(\pi B)^{1-t}\le 1$ (since $\pi \ge B^{-1}$), and in the latter case we end up with $\pi^{1-st} B^{1-t} A^{1-s}\le 1$ (since $\pi \ge (A^{s-1}B^{t-1})^{-1/(st-1)}$). We conclude the result.
	\end{proof}
	
	Finally, we will need the following slightly weaker version of a basic regularization result from \cite{nie2024random}.
	\begin{lem}[Lemma 2.2,~\cite{nie2024random}]\label{lem:regularize}
		If $H$ is an $n$-vertex $3$-partite $3$-graph on $V_1\cup V_2\cup V_3$, then there exists a subgraph $H'\sub H$ and a real-valued vector $\Delta$ indexed by $2^{[3]}$ such that $|H'|\ge  |H|\log(n)^{-8}$ and such that for any set $I\sub [3]$, every $|I|$-element subset $S$ of $\bigcup_{i\in I} V_i$ which is contained in an edge of $H'$ satisfies \[\Delta_I(12 \log n)^{-8} \le \deg_{H'}(S)\le \Delta_I.\]
	\end{lem}
	\section{Proof of the Main Result}
	In this section we use the results of the previous section to prove \Cref{thm:Kst}.  We begin by establishing our main technical supersaturation result \Cref{prop:technicalSupersaturation}, after which we optimize its parameters in \Cref{cor:supersaturation} before finally using this result together with the general results from \cite{nie2024random} on random Tur\'an problems for expansions to finish the proof of \Cref{thm:Kst}.
	
	\subsection{Main Supersaturation Result}
	The main technical supersaturation result we need is the following.  We note that the first part of this statement immediately implies \Cref{prop:vanillaSupersaturation} by taking $\ell=k_0$, $k=m n^{-3+3/s}$, and by looking at the range $k\ge n^{1/st}$ so that the minimum equals the first term.
	
	\begin{prop}\label{prop:technicalSupersaturation}
		For all $t\ge s\ge 3$, there exists a sufficiently large constant $k_0$ such that the following holds: if $H$ is an $n$-vertex 3-graph with $kn^{3-3/s}$ hyperedges such that $k\ge k_0 (12\log n)^{100}$ and if  $\ell$ is a real number such that $k_0\le \ell \le k^{1/3}$, then there exists a collection $\c{H}$ of copies of expansions $K_{s,t}^{(3)}$ in $H$ with 
		\[|\c{H}|=\Om^*(\min\{k^{st}n^{s+(s-2)t},\ell^{-(s-1)(3t-1)}k^{(2s-1)t} n^{s-1+1/s+(s-2)t}\}).\]
		Moreover, if $\ell\ge k^{1/2}n^{-1/s}$
		and  
		\[\tau:= \max\{\ell^{-1},\ell^{\frac{3t-1}{st-1}}k^{\frac{-t}{st-1}}n^{\frac{-(s-2)}{s(st-1)}},n^{\frac{-(s-1)}{s(st-1)}}\}^{s-1}\cdot n^{2-1/s},\]
		then we can further guarantee that $\c{H}$ satisfies  \[(\tau/|H|)^{1-|\sig|} \deg_{\c{H}}(\sig)=O^*\left(\frac{|\c{H}|}{v(\c{H})}\right)\textrm{ for all }\sig\sub V(\c{H})=E(H).\]
		
	\end{prop}
	
	\begin{proof}
		We prove the result only in the case when $H$ is 3-partite, the general case following from the 3-partite case and the fact that every 3-graph contains a 3-partite subgraph with a constant proportion of its edges.  Let $H',\Delta$ be the subgraph and vector guaranteed by \Cref{lem:regularize}, and for ease of notation, we will typically write e.g.\ $\Delta_{1,2}$ instead of $\Delta_{\{1,2\}}$.  Let $G_{i,j}$ denote the shadow graph of $H'$ restricted to $V_i\cup V_j$ and let $m_{i,j}=|G_{i,j}|$.  
		
		In order to apply \Cref{lem:translate}, we will need the following observation.
		\begin{claim}\label{cl:translate}
			Let $d_{i,j}= kn^{3-3/s} m_{i,j}^{-1}(12 \log n)^{-16}$ and $D_{i,j}= kn^{3-3/s} m_{i,j}^{-1}(12 \log n)^{8}$. Then $D_{i,j}\ge d_{i,j} \ge 2(s+t+st)$ and every $e\in G_{i,j}$ satisfies $d_{i,j}\le \deg_{H'}(e)\le D_{i,j}$.
		\end{claim}
		\begin{proof}
			Since $m_{i,j}\le n^2$ and $s\ge 3$, we have 
			\[D_{i,j}\ge d_{i,j}\ge  kn^{1-3/s}(12 \log n)^{-16}\ge  k(12 \log n)^{-16}\ge 2(s+t+st),\]
			where the last step used that $k$ is at least a large constant times $(12\log n)^{16}$.  This establishes the first part of the claim.
			
			For the second part, observe that by definition of $H'$ and $\Delta$, each $e\in G_{i,j}$ has \begin{equation}\Delta_{i,j}(12 \log n)^{-8} \le \deg_{H'}(e)\le \Delta_{i,j}.\label{eq:H'Degree}\end{equation} Thus it suffices to prove 
			\begin{equation}kn^{3-3/s} m_{i,j}^{-1}(12 \log n)^{-8} \le \Delta_{i,j}\le kn^{3-3/s} m_{i,j}^{-1} (12\log n)^{8}.\label{eq:basicRegularize}\end{equation}
			To this end, by the upper bound of \eqref{eq:H'Degree} and the properties of $H'$ we have
			\[m_{i,j}\Delta_{i,j}\ge \sum_{e\in G_{i,j}} \deg_{H'}(e)= |H'|\ge kn^{3-3/s}(12 \log n)^{-8} ,\]
			and similarly
			\[m_{i,j}(12 \log n)^{-8}\Delta_{i,j}\le \sum_{e\in G_{i,j}} \deg_{H'}(e)= |H'|\le kn^{3-3/s},\]
			from which \eqref{eq:basicRegularize} follows, proving the result.
		\end{proof}
		We note for later that an essentially identical argument as above shows that for all $i$, 
		\begin{equation}
			\deg_{H'}(v)\le kn^{3-3/s}|V_i|^{-1}(12 \log n)^{8}\ \ \ \forall v\in V_i.\label{eq:basicRegularization2}
		\end{equation}
		Without loss of generality, we will assume from now on that \begin{equation}|V_1|\ge |V_2|\ge |V_3|.\label{eq:relative}\end{equation}  We now break up our argument into two parts based on whether the shadow  of $H'$ is dense or not.  
		
		\textbf{The Dense Case}.  Here we consider the case that $m_{i,j}\ge \ell |V_i|^{2-1/s}$ for some $i<j$, say with \[m_{i,j}=L |V_i|^{2-1/s}\] for some $L\ge \ell$.  If $k_0$ (and hence $\ell$, and hence $L$) is sufficiently large, then in view of \eqref{eq:relative} we can apply \Cref{prop:dense} to find a collection $\c{G}$ of $K_{s,t}$ in $G_{i,j}$ with 
		\begin{equation}|\c{G}|=\Om(L^{st}|V_i|^s)=\Om(m_{i,j}^{st} |V_i|^{s+t-2st}),\label{eq:denseGsize}\end{equation}
		and such that every $\sig'\sub E(G_{i,j})$ with $|\sig'_v\cap V_j|=a\ge 1$ and $|\sig_v'\cap V_i|=b\ge 1$ satisfies \begin{equation}\deg_{\c{G}}(\sig')= O(( L |V_i|^{1-1/s})^{s-a}( L^s)^{t-b}).\label{eq:denseGcodegree}\end{equation}
		
		By \Cref{cl:translate} and the properties of $\c{G}$ above, we can apply \Cref{lem:translate} with $c_{i,j}=st$ to give a collection $\c{H}$ of $K_{s,t}^{(3)}$'s in $H$.  We aim to show that $\c{H}$ satisfies our desired conditions.  In what follows, we define
		\[M:=k^{st} n^{3st-3t}|V_i|^{s+t-2st}\cdot k^{-1} n^{-3+3/s}.\]\begin{claim}\label{cl:denseHSize}
			The collection $\c{H}$ satisfies
			\[\frac{|\c{H}|}{v(\c{H})}=\Om^*(M),\]
			and 
			\[|\c{H}|=\Om^*(k^{st}n^{s+(s-2)t}).\]
		\end{claim}
		\begin{proof}
			Recall that $d_{i,j}=\Om^*( kn^{3-3/s} m_{i,j}^{-1})$.  By \Cref{lem:translate} and \eqref{eq:denseGsize}, we find
			\begin{align}|\c{H}|&=\Om(m_{i,j}^{st} |V_i|^{s+t-2st}\cdot d_{i,j}^{st})=\Om^*(|V_i|^{s+t-2st}(kn^{3-3/s})^{st})\nonumber\\&=\Om^*(k^{st} n^{3st-3t}|V_i|^{s+t-2st})\label{eq:denseHSize},\end{align}
			and combining this with $v(\c{H})=|H'|\le kn^{3-3/s}$ gives the first bound. The second bound follows from \eqref{eq:denseHSize} and $|V_i|\le n$, which we can apply since the exponent of $|V_i|$ in \eqref{eq:denseHSize} is negative.
		\end{proof}
		It remains now to establish the desired codegree bound\footnote{We only need to prove our codegree bound when we are given the additional hypothesis $\ell\ge k^{1/2}n^{-1/s}$, but this hypothesis will not be needed in the dense case of the argument.} for $\c{H}$. We begin with the following.
		\begin{claim}\label{cl:denseHcodegree}
			If $\sig \sub H$ intersects $V_j$ in $a\ge 1$ vertices and $V_i$ in $b\ge 1$ vertices, then
			\[\deg_{\c{H}}(\sig)=O^*\left((\ell^{-1} k n^{1-2/s})^{1-|\sig|} (\ell n^{1-1/s})^{1-a} (\ell^s)^{1-b}\cdot M\right).\]
		\end{claim}   
		\begin{proof}
			Observe that any $\sig'\in \proj(\sig)$ which does not intersect $V_j,V_i$ in exactly $a,b$ vertices respectively will have $\deg_{\c{G}}(\sig')=0$ (since every edge in a copy of $\c{G}$ lies in $V_i\cup V_j$).  Thus by \Cref{lem:translate} and \eqref{eq:denseGcodegree}, we find
			\begin{align*}\deg_{\c{H}}(\sig)&=O\left( ( L |V_i|^{1-1/s})^{s-a}( L^s)^{t-b}\cdot D_{i,j}^{st-|\sig|}\right)\\ &= O^*\left((L |V_i|^{1-1/s})^{s-a}(L^s)^{t-b}\cdot (kn^{3-3/s} L^{-1}|V_i|^{-2+1/s})^{st-|\sig|}\right)\\ &=O^*\left(L^{s-a+st-sb-st+|\sig|}  |V_i|^{s-1-(1-1/s)a-2st+t+(2-1/s)|\sig|}k^{st-|\sig|}n^{3st-3t-(3-3/s)|\sig|})\right)\\ &= O^*\left( L^{s-a+|\sig|-sb} |V_i|^{-1-(1-1/s)a+(2-1/s)|\sig|}k^{1-|\sig|}n^{-(3-3/s)(|\sig|-1)}\cdot M\right)\\ &= O^*\left(  L^{s-a+|\sig|-sb}  |V_i|^{(2-1/s)(|\sig|-1)-(1-1/s)(a-1)}k^{1-|\sig|}n^{-(3-3/s)(|\sig|-1)}\cdot M\right),\end{align*}
			with this second equality using $D_{i,j}=O^*(kn^{3-3/s}\cdot m_{i,j}^{-1})=O^*(kn^{3-3/s} \cdot L^{-1} |V_i|^{-2+1/s})$ and the fourth equality using $M=k^{st-1} n^{3st-3t-3+3/s}|V_i|^{s+t-2st}$.  Using that $|\sig|\le ab$ (since $\sig$ intersects $V_j,V_i$ in $a,b$ vertices), we see that the exponent of $L$ above
			equals \[s-a+|\sig|-sb\le(s-a)(1-b)\le 0.\] Similarly the exponent for $|V_i|$ is non-negative since $|\sig|\ge a$.  Thus the expression above is maximized when $L\ge \ell$ is as small as possible and $|V_i|\le n$ is as large as possible, and plugging these values in and rearranging slightly yields  
			
			\[\deg_{\c{H}}(\sig)=O^*(\ell^{(|\sig|-1)+(1-a)+s(1-b)} k^{1-|\sig|} n^{-(1-2/s)(|\sig|-1)-(1-1/s)(a-1)}\cdot M),\]
			equalling the desired bound.
		\end{proof}
		Now consider any $\sig \sub H$ which intersects $V_j$ in $a\ge 1$ vertices and $V_i$ in $b\ge 1$ and define
		\[\pi:= \max\{\ell^{-s+1},n^{\frac{-(s-1)^2}{s(st-1)}}\}\cdot \ell^{-1}.\]
		Recalling that
		\[\tau= \max\{\ell^{-1},\ell^{\frac{3t-1}{st-1}}k^{\frac{-t}{st-1}}n^{\frac{-(s-2)}{s(st-1)}},n^{\frac{-(s-1)}{s(st-1)}}\}^{s-1}\cdot n^{2-1/s},\]
		we see that $\tau\ge \ell  n^{2-1/s} \pi$ (namely by dropping the middle term in the maximum of $\tau$).  This together with $|H|=kn^{3-3/s}$ and \Cref{cl:denseHcodegree} gives
		\begin{align*}(\tau/|H|)^{1-|\sig|} \cdot \deg_{\c{H}}(\sig)\le (\ell k^{-1} n^{-1+2/s}\pi )^{1-|\sig|} \cdot \deg_{\c{H}}(\sig)&=O^*(\pi^{1-|\sig|} (\ell n^{1-1/s})^{1-a} (\ell^s)^{1-b}\cdot M)\\ &=O^*(\pi^{1-ab} (\ell n^{1-1/s})^{1-a} (\ell^s)^{1-b}\cdot M),\end{align*}
		with this last step using $|\sig|\le ab$ and $\pi\le 1$. Note that if we can show  the expression above is at most $O^*(M)$, then by \Cref{cl:denseHSize} it will be at most $O^*(|\c{H}|/v(\c{H}))$, completing the proof in this case.  Thus, it suffices to show \[\pi^{1-ab} (\ell n^{1-1/s})^{1-a} (\ell^s)^{1-b}\le 1,\] and to do this, it will suffice to show that the conditions of \Cref{lem:optimize} are satisfied with $A=\ell n^{1-1/s}$ and $B=\ell^s$; i.e.\ to show $A\ge B$ and $\pi\ge B^{-1},(A^{s-1}B^{t-1})^{\frac{-1}{st-1}}$. 
		
		Observe that we have $\ell n^{1-1/s}\ge \ell^s$ since $\ell \le k^{1/3}\le n^{1/s}$ (this last inequality holding because $kn^{3-3/s}=|H|\le n^3$), we have $\pi\ge \ell^{-s}$ due to the $\ell^{-s+1}$ term in $\pi$'s maximum, and similarly we have 
		\[ (\ell^{s-1} n^{(s-1)^2/s}\cdot \ell^{st-s})^{\frac{-1}{st-1}}=\ell^{-1} n^{\frac{-(s-1)^2}{s(st-1)}}\le \pi.\]
		Thus \Cref{lem:optimize} applies, completing the proof in this case.
		
		\textbf{The Sparse Case}.  We now assume that  $m_{i,j}< \ell |V_i|^{2-1/s}$ for all $i<j$, which in particular implies $m_{i,j}\le \ell |V_1|^{2-1/s}$ for all $i,j$ in view of \eqref{eq:relative}.
		
		\begin{claim}
			$H'$ satisfies the conditions of \Cref{prop:sparse} with $K:=k (12\log n)^{-24}$ provided $k_0$ is sufficiently large.
		\end{claim}
		\begin{proof}
			Condition (a), saying $|V_i|\le n$, is trivial since $H'\sub H$ and $H$ has $n$ vertices.  Condition (b), saying $|H'|\ge K n^{3-3/s}(12\log n)^{16}$, follows from \Cref{lem:regularize} since this implies \[|H'|\ge  |H|(12\log n)^{-8}=K n^{3-3/s}(12\log n)^{16}.\]
			Condition (c), saying $m_{i,j}\le \ell |V_1|^{2-1/s}$, follows from definition of being in the sparse case. Condition (d), saying $\del^{-1}\le \ell \le \del K^{(s-1)/(2s-3)}$, follows from \[k_0\le \ell\le k^{1/3}\le k^{1/3}\cdot k^{1/6}(12\log n)^{-12}=K^{1/2}\le \del K^{(s-1)/(2s-3)},\]
			where the third and last inequality used the hypothesis $k\ge (12\log n)^{100}$ and that $n$ is sufficiently large.
			
			Condition (e), saying pairs in $V_i\cup V_j$ are contained in at most $K n^{3-3/s}m_{i,j}^{-1}(12 \log n)^{32}=k n^{3-3/s}m_{i,j}^{-1}(12 \log n)^{8}$ edges of $H'$, follows from \Cref{cl:translate}.   Similarly, for condition (f), saying that each vertex in $V_1$ has at most $\ell |V_1|^{1-1/s}(12 \log n)^8$ neighbors in $V_2$, we observe that \Cref{lem:regularize} together with \eqref{eq:basicRegularization2} and \Cref{cl:translate} implies each $v_1\in V_1$ has degree at most $kn^{3-3/s}|V_1|^{-1} (12 \log n)^{8}$ in $H'$ and that each $v_2\in V_2$ which is contained in an edge with $v_1$ is contained in at least $kn^{3-3/s} m_{1,2}^{-1}(12\log n)^{-8}$ edges of $H'$ with $v_1$, so in total the number of neighbors of each $v_1$ in $V_2$ can be at most $m_{1,2} |V_1|^{-1}(12 \log n)^{16}\le \ell |V_1|^{1-1/s}(12\log n)^{16}$, with this last inequality using that we are in the sparse case.
		\end{proof}
		Let $\c{G}$ be the collection of $K_{s,t}$'s guaranteed by \Cref{prop:sparse}.  By \Cref{cl:translate},  we can apply \Cref{lem:translate} to obtain a collection $\c{H}$ of expansions $K_{s,t}^{(3)}$ with $c_{1,3}=t$ and $c_{2,3}=(s-1)t$, so it remains to show that this collection satisfies the conditions of the theorem.  To this end, define
		\[\widetilde{M}:=\ell^{s-1} n^{s-1+1/s}\left(\ell^{3-2s}k^{2s-1} n^{3s-3} m_{1,2}^{-1} m_{2,3}^{-(s-1)}\right)^t\cdot k^{-1} n^{-3+3/s}.\]
		\begin{claim}\label{cl:sparseHSize}
			The collection $\c{H}$ satisfies
			\[\frac{|\c{H}|}{v(\c{H})}=\Om^*(\widetilde{M}),\]
			and 
			\[|\c{H}|=\Om^*(\ell^{-(s-1)(3t-1)}k^{(2s-1)t}n^{s-1+1/s+(s-2)t}).\]
		\end{claim}
		\begin{proof}    
			By using \Cref{lem:translate} together with $|\c{G}|=\Om^*(\ell^{s-1} n^{s-1+1/s}(\ell^{3-2s}K^{s-1})^t)$ from \Cref{prop:sparse}, $d_{i,j}=\Om^*(kn^{3-3/s} m_{i,j}^{-1})$, and $K=\Om^*(k)$; we find
			\begin{align}|\c{H}|&=\Om(|\c{G}|\cdot d_{1,3}^t d_{2,3}^{(s-1)t})=\Om^*(\ell^{s-1}n^{s-1+1/s}(\ell^{3-2s}K^{s-1})^t\cdot (kn^{3-3/s})^{st} m_{1,3}^{-t} m_{2,3}^{-(s-1)t})\nonumber\\ &=\Om^*(\ell^{s-1}n^{s-1+1/s}(\ell^{3-2s} k^{2s-1}n^{3s-3}m_{1,3}^{-1} m_{2,3}^{-(s-1)})^t)\label{eq:tempHBound},\end{align}
			giving the first bound after noting $v(\c{H})\le kn^{3-3/s}$. Using  \eqref{eq:tempHBound} and $m_{1,3} m_{2,3}^{s-1}\le (\ell n^{2-1/s})^s$ by hypothesis of being in the sparse case, we find
			\[|\c{H}|=\Om^*(\ell^{s-1}n^{s-1+1/s}(\ell^{3-3s}k^{2s-1}n^{3s-3-2s+1})^t)=\Om^*(\ell^{-(s-1)(3t-1)}k^{(2s-1)t}n^{s-1+1/s+(s-2)t}).\]
		\end{proof}
		With this it remains only to conclude our desired codegree bounds under the additional hypothesis $\ell\ge k^{1/2}n^{-1/s}$ for $\c{H}$, which we do by applying \Cref{lem:translate}.  For this we establish the following.
		\begin{claim}
			If $\sig\sub H$ and $\sig'\in \proj(\sig)$ with $a=|\sig'_v\cap  (V_1\cup V_2)|$ and $b=|\sig'_v\cap V_3|$, then
			\[ \deg_{\c{G}}(\sig')\prod D_{i,j}^{c_{i,j}-|\sig'\cap G_{i,j}|}=O^*\left( (\ell^{-1}k n^{1-2/s})^{1-|\sig'|}( \ell^{-1} kn^{1-2/s})^{1-a} (\ell^{3-2s}k^{s-1})^{1-b} \cdot \widetilde{M} \right).\]
		\end{claim}
		\begin{proof}
			If $a=0$ or $b=0$ then $\deg_{\c{G}}(\sig')=0$ (as every edge of each $K_{s,t}$ in $\c{G}$ intersects both $V_1\cup V_2$ and $V_3$), so we may assume $a,b\ge 1$.
			
			Consider the subcase that $\sig'$ intersects $V_1$, i.e.\ that $|\sig'\cap V_1|\ge 1$.   In this case using $c_{1,3}=t,\ c_{2,3}=(s-1)t$, and $D_{i,j}=O^*(kn^{3-3/s} m_{i,j}^{-1})$, we find \begin{align}\prod D_{i,j}^{c_{i,j}-|\sig'\cap G_{i,j}|}&= D_{1,3}^t D_{2,3}^{(s-1)t} \cdot \prod D_{i,j}^{-|\sig'\cap G_{i,j}|}\nonumber\\ &= O^*\left((kn^{3-3/s})^{st} m_{1,3}^{-t}m_{2,3}^{-(s-1)t} \cdot \prod D_{i,j}^{-|\sig'\cap G_{i,j}|}\right)\nonumber \\ &= O^*\left((kn^{3-3/s})^{st} m_{1,3}^{-t}m_{2,3}^{-(s-1)t} \cdot D_{1,3}^{-1} (k n^{3-3/s}\cdot \ell^{-1} n^{-2+1/s})^{1-|\sig'|}\right),\label{eq:DProduct}\end{align}
			where this last equality used that $|\sig'\cap G_{1,3}|\ge 1$ by the hypothesis of $|\sig'\cap V_1|,b\ge 1$ to guarantee a $D_{1,3}$ term in the product, and that \[D_{i,j}\ge k n^{3-3/s} \cdot m_{i,j}^{-1}\ge k n^{3-3/s}\cdot \ell^{-1} n^{-2+1/s}\] for all $i,j$ to bound the remaining terms of the product.   Again using $D_{1,3}\ge  k n^{3-3/s} m_{1,3}^{-1}$ in \eqref{eq:DProduct} gives in total that
			\[\prod D_{i,j}^{c_{i,j}-|\sig'\cap G_{i,j}|}=O^*\left(m_{1,3}\cdot m_{1,3}^{-t}m_{2,3}^{-(s-1)t} (kn^{3-3/s})^{st-1} (\ell^{-1}k n^{1-2/s})^{1-|\sig'|}\right).\]
			A nearly identical argument gives that if  $\sig'$ is disjoint from $V_1$ then
			\[\prod D_{i,j}^{c_{i,j}-|\sig'\cap G_{i,j}|}=O^*\left(  m_{2,3}\cdot m_{1,3}^{-t}m_{2,3}^{-(s-1)t} (kn^{3-3/s})^{st-1} (\ell^{-1}k n^{1-2/s})^{1-|\sig'|}\right).\]
			By \Cref{prop:sparse} we have
			\[\deg_{\c{G}}(\sig')\le \phi(\sig'_v)=\frac{\ell^{s-1} n^{s-1+1/s}(\ell^{3-2s}K^{s-1})^t}{\del m_{I,3}(\del \ell^{-1} K n^{1-2/s})^{a-1}(\del \ell^{3-2s} K^{s-1})^{b-1}}\]
			where $I=1$ if $\sig'$ intersects $V_1$ and $I=2$ otherwise.  Combining this with the bounds for $\prod D_{i,j}^{c_{i,j}-|\sig'\cap G_{i,j}|}$ above, we see that regardless of whether $\sig'$ intersects $V_1$, we have that $\deg_{\c{G}}(\sig')\cdot \prod D_{i,j}^{c_{i,j}-|\sig'\cap G_{i,j}|}$ is at most 
			\[O^*\left(\ell^{s-1} n^{s-1+1/s}(\ell^{3-2s}k^{s-1})^t (\ell^{-1} kn^{1-2/s})^{1-a} (\ell^{3-2s}k^{s-1})^{1-b}  \cdot  m_{1,2}^{-t}m_{2,3}^{-(s-1)t} (kn^{3-3/s})^{st-1} (\ell^{-1}k n^{1-2/s})^{1-|\sig'|}\right),\]
			and rearranging with respect to $\widetilde{M}=\ell^{s-1} n^{s-1+1/s}(\ell^{3-2s}k^{2s-1} n^{3s-3})^t m_{1,2}^{-t}m_{2,3}^{-(s-1)t}\cdot k^{-1} n^{-3+3/s}$ shows that the above equals $O^*( (\ell^{-1}k n^{1-2/s})^{1-|\sig'|}( \ell^{-1} kn^{1-2/s})^{1-a} (\ell^{3-2s}k^{s-1})^{1-b} \cdot \widetilde{M} )$ as desired.        \end{proof}
		Define \[\widetilde{\pi}:=\max\{\ell^{1-s},\ell^{\frac{(3t-1)(s-1)}{st-1}}k^{\frac{-t(s-1)}{st-1}} n^{\frac{-(s-2)(s-1)}{s(st-1)}}\}\ell^{-1},\] noting that $\tau\ge \ell  n^{2-1/s}\widetilde{\pi}$ since $\widetilde{\pi}$ has fewer terms in its maximum compared to $\tau$.  Using this together with \Cref{lem:translate} and the claim above, we see that for any $\sig$ intersecting $V_1\cup V_2$ in $a$ vertices and $V_3$ in $b$ vertices, we have 
		\[(\tau/|H|)^{1-|\sig|} \deg_{\c{H}}(\sig)=O^*(\widetilde{\pi}^{1-|\sig|}(\ell^{-1}kn^{1-2/s})^{1-a}(\ell^{3-2s}k^{s-1})^{1-b}\cdot \widetilde{M} )=O^*( \widetilde{\pi}^{1-ab}(\ell^{-1}kn^{1-2/s})^{1-a}(\ell^{3-2s}k^{s-1})^{1-b}\cdot \widetilde{M}),\]
		where here we used $|\sig'|=|\sig|$ for $\sig'\in \proj(\sig)$ and that $|\sig|\le ab$ and $\widetilde{\pi}\le 1$.   By \Cref{cl:sparseHSize}, to show that this is at most $O^*(|\c{H}|/v(\c{H}))$, it suffices to show it is at most $O^*(\widetilde{M})$, i.e. that \[\widetilde{\pi}^{1-ab}(\ell^{-1}kn^{1-2/s})^{1-a}(\ell^{3-2s}k^{s-1})^{1-b}\le 1.\]  For this, it suffices to show that the conditions of \Cref{lem:optimize} are met.  And indeed, we have $\ell^{-1}kn^{1-2/s}\ge \ell^{3-2s}k^{s-1}$ (as this is equivalent to $\ell\ge k^{1/2}n^{-1/s}$), that $\widetilde{\pi} \ge \ell^{2s-3} k^{1-s}$ (since $\widetilde{\pi} \ge \ell^{-s}$ and $\ell\le k^{1/3}$), and that 
		\begin{align*}&(\ell^{-s+1} k^{s-1}n^{\frac{-(s-2)(s-1)}{s}}\cdot \ell^{(3-2s)(t-1)}k^{(s-1)(t-1)})^{-1/(st-1)}\\=&\ell^{-1}\cdot \ell^{\frac{(3t-1)(s-1)}{st-1}} k^{\frac{-t(s-1)}{st-1}} n^{\frac{-(s-2)(s-1)}{s(st-1)}}\le \widetilde{\pi}.\end{align*}
		Thus \Cref{lem:optimize} applies, giving the desired result and completing the proof.
	\end{proof}
	
	\subsection{Optimizing the Parameters}
	Optimizing the choice of $\ell$ in \Cref{prop:technicalSupersaturation} gives the following.
	
	\begin{cor}\label{cor:supersaturation}
		For all $t\ge s\ge 3$, there exists a sufficiently large constant $K_0$ such that the following holds.  Let $H$ be an $n$-vertex 3-graph with $kn^{3-3/s}$ hyperedges such that $k\ge K_0 (\log n)^{100}$.  There exists a non-empty collection $\c{H}$ of copies of expansions $K_{s,t}^{(3)}$ in $H$ 
		such that 
		\[\tau=\max\l\{k^{-1/3},~k^{\frac{-t}{st+3t-2}}n^{\frac{-(s-2)}{s(st+3t-2)}},n^{\frac{-(s-1)}{s(st-1)}}\r\}^{s-1}\cdot n^{2-\frac{1}{s}}\]
		satisfies  $(\tau/|H|)^{1-|\sig|} \deg_{\c{H}}(\sig)=O^*\left(\frac{|\c{H}|}{v(\c{H})}\right)$ for all $\sig\sub V(\c{H})=E(H)$.
	\end{cor}
	%Note that these cases break up into the intervals $[1, n^{\frac{3(s-2)}{s(st-2)}}],\ [n^{\frac{3(s-2)}{s(st-2)}}, n^{\frac{4st-s-3t}{st(st-1)}}],\ [n^{\frac{4st-s-3t}{st(st-1)}},n^{3/s}]$
	
	\begin{proof}
		Given $k,\ell,n$, we define
		\[\mu:=\max\{\ell^{-1},\ell^{\frac{3t-1}{st-1}}k^{\frac{-t}{st-1}}n^{\frac{-(s-2)}{s(st-1)}},n^{\frac{-(s-1)}{s(st-1)}}\},\]
		and
		\[\pi:=\max\l\{k^{-1/3},~k^{\frac{-t}{st+3t-2}}n^{\frac{-(s-2)}{s(st+3t-2)}},n^{\frac{-(s-1)}{s(st-1)}}\r\}.\]
		\begin{claim}
			To prove the result, it suffices to show that for all $k\ge K_0 (\log n)^{100}$ that there exists an $\ell$ satisfying
			\begin{equation}\max\{k_0,k^{1/2}n^{-1/2s}\}\le \ell\le k^{1/3} \label{eq:ell}\end{equation}
			such that
			\[\mu\le \pi.\]
		\end{claim}
		\begin{proof}
			Consider some $k\ge K_0(\log n)^{100}$ and let $\ell$ be as guaranteed by the claim.  By \Cref{prop:technicalSupersaturation}, there exists a non-empty collection $\c{H}$ of copies of $K_{s,t}^{(3)}$ in $H$ such that 
			\[\widetilde{\tau}:=\mu^{s-1} n^{2-1/s}\le \pi^{s-1} n^{2-1/s}=\tau \]
			satisfies
			\[(\widetilde{\tau}/|H|)^{1-|\sig|} \deg_{\c{H}}(\sig)=O^*\left(\frac{|\c{H}|}{v(\c{H})}\right)\]
			for all $\sig \sub H$.  Since $\widetilde{\tau}\le \tau$, we have $\tau^{1-|\sig|}\le \widetilde{\tau}^{1-|\sig|}$, so the bound above continues to hold for $\tau$, proving the result.
		\end{proof}
		Given $k,n$, we will say that $\ell$ is \textit{valid} if it satisfies \eqref{eq:ell}.  It remains then to show for each $k\ge K_0(\log n)^{100}$ there exists a valid $\ell$ such that $\mu\le \pi$.
		
		\textbf{Intuition}.   We will always consider $\ell$ which are either at the cusp of being valid (i.e. $\ell=k^{1/2}n^{-1/2s}$ or $\ell=k^{1/3}$), or which are such that the first $\mu$ term $\ell^{-1}$ equals one of the latter two terms of $\mu$ (i.e. $\ell=k^{\frac{t}{st+3t-2}}n^{\frac{s-2}{s(st+3t-2)}}$ or $\ell=n^{\frac{s-1}{s(st-1)}}$).  
		
		In order to show our choices for $\ell$ are valid, we will need to verify inequalities of the form $f(k,n)\le g(k,n)$ for all $k\in [n^{\alpha},n^{\beta}]$ with $f,g$  polynomials in $k,n$.  We will typically do this by observing that, for example, if the exponent for $k$ in $f$ is larger than it is in $g$, then it suffices to verify this inequality at the largest possible value of $k=n^{\beta}$ (since the exponent condition implies this holds for all smaller values of $k$).  In fact, in most cases we will have $f(k,n)=g(k,n)$ at $k=n^{\beta}$, in particular giving the desired inequality.  In our proof we will often omit spelling out all of these details each time and only say things like ``because equality holds for the largest possible value of $k$.''

		To show $\mu\le \pi$, we will typically show that $\mu$ equals the maximum of one or two of the terms in the maximum used to define $\pi$ (and hence $\pi$ must be at least as large as $\mu$).
		
		\textbf{Formal Details}.  First consider $k\in [K_0(\log n)^{100}, n^{\frac{3(s-2)}{s(st-2)}}]$.  In this case we take $\ell=k^{1/3}$ (which is always valid for $k\le n^{3/s}$ provided $K_0$ is sufficiently large in terms of $k_0$), which yields
		\[\mu=\max\{k^{-1/3},k^{\frac{-1}{3(st-1)}}n^{\frac{-(s-2)}{s(st-1)}},n^{\frac{-(s-1)}{s(st-1)}}\}.\]
		Observe that $k^{-1/3}\ge k^{\frac{-1}{3(st-1)}}n^{\frac{-(s-2)}{s(st-1)}}$ for $k\le n^{\frac{3(s-2)}{s(st-2)}}$ (since this is equivalent to saying $n^{\f{s-2}{s(st-1)}}\ge k^{1/3-1/3(st-1)}=k^{\frac{st-2}{3(st-1)}}$ and equality holds when $k$ is the largest possible value).  Using this with the expression for $\mu$ above and the definition of $\pi$ gives
		\[\mu= \max\{k^{-1/3},n^{\frac{-(s-1)}{s(st-1)}}\}\le \pi,\]
		giving the desired result in this case.
		
		We next consider $k\in [n^{\frac{3(s-2)}{s(st-2)}}, n^{\frac{4st-s-3t}{st(st-1)}}]$, and here we take $\ell=k^{\frac{t}{st+3t-2}}n^{\frac{s-2}{s(st+3t-2)}}$.  We claim that this $\ell$ is valid.  Indeed, $\ell\ge k_0$ holds since $\ell$ is the product of two (large in terms of $K_0$) positive numbers raised to positive powers.  We have $k\ge \ell^3$ for $k\ge n^{\frac{3(s-2)}{s(st-2)}}$ (because this is equivalent to saying $n^{\frac{3(s-2)}{s(st+3t-2)}}\le k^{1-\frac{3t}{st+3t-2}}=k^{\frac{st-2}{st+3t-2}}$ with equality holding at the smallest possible value of $k$), and we have $\ell^2\ge k n^{-1/s}$ (since at the largest possible value of $k$ we have $k n^{-1/s}=n^{\frac{(4-t)st-s-2t}{st(st-1)}}\le 1\le \ell^2$ for $t\ge s\ge 3$).  Plugging in this value of $\ell$ gives
		\[\mu=\max\{k^{\frac{-t}{st+3t-2}}n^{\frac{-(s-2)}{s(st+3t-2)}},k^{\frac{-t}{st+3t-2}}n^{\frac{-(s-2)}{s(st+3t-2)}},n^{\frac{-(s-1)}{s(st-1)}}\}\le \pi,\]
		giving the desired result in this range.
		
		We next consider $k\in [n^{\frac{4st-s-3t}{st(st-1)}}, n^{\frac{st+2s-3}{s(st-1)}}]$ and take $\ell=n^{\frac{s-1}{s(st-1)}}$, noting that $\ell\ge k^{\frac{1}{2}}n^{-\frac{1}{2s}}$ since $k\le n^{\frac{st+2s-3}{st-1}}$ (because equality holds at the largest possible value of $k$), and that $k\ge \ell^3$ (as for the smallest possible value of $k$ this is equivalent to saying $\frac{4st-s-3t}{t}\ge \frac{3st-3t}{t}=3(s-1)$).  With this, we find
		\[\mu=\max\{n^{\frac{-(s-1)}{s(st-1)}},n^{\frac{(3t-1)(s-1)}{s(st-1)^2}}k^{\frac{-t}{st-1}}n^{\frac{-(s-2)}{s(st-1)}},n^{\frac{-(s-1)}{s(st-1)}}\}.\]
		Note that this first term is at least the second term for $k\ge n^{\frac{4st-s-3t}{st(st-1)}}$ (since it holds at the smallest possible value of $n^{\frac{4st-s-3t}{st(st-1)}}=n^{\frac{1}{st}+\frac{(3t-1)(s-1)}{st(st-1)}}$), giving the desired result.
		
		Finally we consider $k\in [n^{\frac{st+2s-3}{s(st-1)}}, n^{3/s}]$.  Here we take $\ell=k^{1/2} n^{-1/2s}$, which is easily seen to be valid (the condition $\ell\le k^{1/3}$ is equivalent to $k\le n^{3/s}$, and the condition $\ell\ge k_0$ is equivalent to $k$ being at least a large constant times $n^{1/s}$,  which follows from $st+2s-3> st-1$).  Here we have
		\[\mu=\max\{k^{-1/2} n^{1/2s},k^{\frac{t-1}{2(st-1)}}n^{\frac{-3t-2s+5}{2s(st-1)}},n^{\frac{-(s-1)}{s(st-1)}}\}.\]
		We have $n^{\frac{-(s-1)}{s(st-1)}}\ge k^{-1/2} n^{1/2s}$ for $k\ge n^{\frac{st+2s-3}{s(st-1)}}$ (with equality at the smallest possible value), and $n^{\frac{-(s-1)}{s(st-1)}}\ge k^{\frac{t-1}{2(st-1)}}n^{\frac{-3t-2s+5}{2s(st-1)}}$ for $k\le n^{3/s}$ (with equality holding at the largest possible value).  Thus $\mu=n^{\frac{-(s-1)}{s(t-1)}}\le \pi$, proving that $\mu\le \pi$ in all cases, finishing the proof.
	\end{proof}
	
	\subsection{Finishing the Proof}\label{sub:finish}
	
	To complete our proof, we recall some definitions and general results from \cite{nie2024random}.  We begin with a general bound for the random Tur\'an number, for which we recall that if $F$ is an $r$-graph $F$ with at least 2 edges, then we denote its \emph{$r$-density} by
	$$
	d_r(F):=\max_{F'\subseteq F,\ |F'|\ge 2}\l\{\frac{|F'|-1}{v(F')-r}\r\}.
	$$
	
	\begin{lem}[Proposition 2.5,~\cite{nie2024random}]\label{lem:generalLowerRandomTuran}
		Let $F$ be an $r$-graph with $\Delta(F)\ge 2$.  If $n^{-r}\ll p \ll n^{-\frac{1}{d_r(F)}}$, then a.a.s.
		\[\ex(G_{n,p}^r,F)=(1+o(1))p {n\choose r}.\]
		If $p\gg  n^{-\frac{1}{d_r(F)}}$, then a.a.s.
		\[\ex(G_{n,p}^r,F)\ge  \max\{\Omega(p\cdot \ex(n,F)),n^{r-\frac{1}{d_r(F)}}(\log n)^{-1}\}.\]
		%and if $p\cdot \ex(n,F)\to \infty$, then a.a.s.
		%\[\ex(G_{n,p}^r,F)= \Omega(p\cdot \ex(n,F)).\]
	\end{lem}
	
	We next state a general upper bound for random Tur\'an numbers, for which we need the following precise notion of what it means for an $r$-graph to have a balanced supersaturation result.
	\begin{defn}\label{def:balanced}
		Given positive functions $M=M(n),\ \gamma=\gamma(n)$, and $\tau=\tau(n,m)$, we say that an $r$-graph $F$ is \emph{$(M,\gamma,\tau)$-balanced} if for every $n$-vertex $m$-edge $r$-graph $H$ with $n$ sufficiently large and $m\ge M(n)$, there exists a non-empty collection $\c{H}$ of copies of $F$ in $H$ (which we view as an $|F|$-uniform hypergraph on $E(H)$) such that, for all integers $1\le i\le |F|$,
		\begin{equation}\label{equation:Delta}
			\Delta_i(\c{H})\le \frac{\gamma(n)|\c{H}|}{m}\l(\frac{\tau(n,m)}{m}\r)^{i-1}.        
		\end{equation}
	\end{defn}
	In other words, $F$ is $(M,\gamma,\tau)$-balanced if for every $r$-graph with sufficiently many vertices and a large number of edges (i.e.\ $m\ge M(n)$), we can find a collection of copies of $F$ which are spread out in the sense of \eqref{equation:Delta}.  This definition together with an application of the method of hypergraph containers~\cite{balogh2015independent,saxton2015hypergraph} gives the following. 
	
	\begin{prop}[Proposition 2.6,~\cite{nie2024random}]\label{Lemma:General Random Turan}
		Let $F$ be an $r$-graph with $r\ge2$.  If there exists a  $C>0$ and positive functions $M=M(n)$ and $\tau=\tau(n,m)$ such that 
		\begin{itemize}
			\item[(a)] $F$ is $(M,(\log n)^{C}, \tau)$-balanced, and
			\item[(b)] For all sufficiently large $n$ and $m\ge M(n)$, the function $\tau(n,m)$ is non-increasing with respect to $m$ and satisfies $\tau(n,m)\ge 1$, 
		\end{itemize}
		then there exists $C'\ge 0$ such that for all sufficiently large $n$, $m\ge M(n)$, and $0<p\le 1$ with $pm\rightarrow\infty$ as $n\rightarrow \infty$, we have a.a.s.
		$$
		\ex(G^r_{n,p},F)\le \max\l\{C'pm,\tau(n,m)(\log n)^{C'}\r\}.
		$$
	\end{prop}
	
	The last ingredient we need is a ``lifting'' result which says that if some $r_0$-graph $F$ is sufficiently balanced, then so is every expansion $F^{(r)}$.
	
	\begin{thm}[Theorem 2.7,~\cite{nie2024random}]\label{lemma:BSviaSHADOW}
		Let $r>r_0\ge 2$ and let $F$ be an $r_0$-graph such that there exist positive functions $M_{r_0}=M_{r_0}(n)$, $\gamma_{r_0}=\gamma_{r_0}(n)$ and $\tau_{r_0}=\tau_{r_0}(n,m)$ with the following properties:
		\begin{itemize}
			\item [(a)] $F$ is $(M_{r_0}, \gamma_{r_0}, \tau_{r_0})$-balanced,
			\item[(b)] For sufficiently large $n$, the function $M_{r_0}(n)n^{-\frac{r-r_0}{r-1}}$ is non-decreasing with respect to $n$,
			\item[(c)] For sufficiently large $n$, the function  $\gamma_{r_0}(n)$ is non-decreasing with respect to $n$.
			\item [(d)] For sufficiently large $n$,  $m\ge M_{r_0}(n)$, and any real $x\ge 1$, the function $\tau_{r_0}(n,m)$ is non-increasing with respect to $m$ and $\tau_{r_0}(nx,mx^{\frac{r-r_0}{r-1}})$ is non-decreasing with respect to $x$.
		\end{itemize}
		Then there exists a sufficiently large constant $C$ depending only on $F$ and $r$ such that for the functions
		\begin{equation*}
			M_r(n):=\max\l\{M_{r_0}(n)^{\frac{r-1}{r_0-1}}n^{-\frac{r-r_0}{r_0-1}},~n^{r-1}\r\}(\log n)^{C},
		\end{equation*}
		
		\begin{equation*}
			\gamma_r(n):=\gamma_{r_0}(n)(\log n)^{C},
		\end{equation*}
		and
		\begin{equation*}
			\tau_r(n,m):=\tau_{r_0}\l(n,n^{\frac{r-r_0}{r-1}}m^{\frac{r_0-1}{r-1}}(\log n)^{-C}\r)(\log n)^{C},
		\end{equation*}
		we have that $F^{(r)}$ is $(M_r,\gamma_r,\tau_r)$-balanced.
	\end{thm}
	
	In order to use this result, it will be useful to restate \Cref{cor:supersaturation} in terms of the language from \Cref{def:balanced}.
	\begin{cor}\label{cor:supersaturation'}
		For all $t\ge s\ge 3$, there exist sufficiently large constant $C_3$ such that $K^{(3)}_{s,t}$ is \\$(n^{3-3/s}(\log n)^{C_3},(\log n)^{C_3},\tau_3)$-balanced where
		$$
		\tau_3(n,m)=\max\l\{m^{-1/3}n^{1-1/s},~m^{\frac{-t}{st+3t-2}}n^{\frac{3st-3t-s+2}{s(st+3t-2)}},~n^{\frac{-(s-1)}{s(st-1)}}\r\}^{s-1}\cdot n^{2-\frac{1}{s}}.
		$$
	\end{cor}
	
	Using \Cref{lemma:BSviaSHADOW} we can now ``lift'' \Cref{cor:supersaturation'} to higher uniformity as follows.
	
	\begin{thm}\label{thm:supersaturation-r}
		For all $r\ge 3$ and $t\ge s\ge 3$, there exist sufficiently large constant $C_r$ such that $K^{(r)}_{s,t}$ is $(n^{\alpha_r}(\log n)^{C_r},(\log n)^{C_r},\tau_r)$-balanced where
		$$
		\alpha_r=\max\l\{r-\frac{3(r-1)}{2s},~r-1\r\},
		$$
		and
		$$
		\tau_r(n,m)=\max\l\{m^{-\frac{2}{3(r-1)}}n^{\frac{2r}{3(r-1)}-\frac{1}{s}},~m^{\frac{-2t}{(r-1)(st+3t-2)}}n^{\frac{2rt}{(r-1)(st+3t-2)}-\frac{3t+s-2}{s(st+3t-2)}},~n^{\frac{-(s-1)}{s(st-1)}}\r\}^{s-1} n^{2-\frac{1}{s}}(\log n)^{C_r}.
		$$
	\end{thm}
	
	\begin{proof}
		Let $C_3$ be the constant guaranteed by \Cref{cor:supersaturation'}. We want to apply \Cref{lemma:BSviaSHADOW} with $F=K^{(3)}_{s,t}$, $M_3(n)=(\log n)^{C_3}n^{3-3/s}$, $\gamma_3(n)=(\log n)^{C_3}$, and
		$$
		\tau_3(n,m)=\max\l\{m^{-1/3}n^{1-1/s},~m^{\frac{-t}{st+3t-2}}n^{\frac{3st-3t-s+2}{s(st+3t-2)}},~n^{\frac{-(s-1)}{s(st-1)}}\r\}^{s-1} n^{2-\frac{1}{s}}.
		$$
		To this end, we need to check the conditions (a)-(d) in \Cref{lemma:BSviaSHADOW}. Condition (a) is guaranteed by \Cref{cor:supersaturation'}. It is easy to check that conditions (b) and (c) hold. To check condition (d), note that the coefficient of $m$ in $\tau_3(m,n)$ is negative and that the coefficient of $x$ in 
		$\tau_3(nx,mx^{\frac{r-3}{r-1}})$ is at least 
		
		\[
		\begin{aligned}
			\min\l\{-\frac{1}{3} \cdot \frac{r-3}{r-1}+1-\frac{1}{s},~\frac{-t}{st+3t-2}\cdot\frac{r-3}{r-1}+\frac{3st-3t-s+2}{s(st+3t-2)},~\frac{-(s-1)}{s(st-1)}\r\}(s-1)+2-\frac{1}{s}.
		\end{aligned}\]
		
		Using $\frac{r-3}{3r-3}<1$, $\frac{r-3}{r-1}<1$, and $\frac{s-1}{s}<1$, we see that this exponent for $x$ is strictly larger than
		\[
		\begin{aligned}
			\min\l\{-\frac{1}{s},~\frac{2st-3t-s+2}{s(st+3t-2)},~\frac{-1}{st-1}\r\}(s-1)+1>0,
		\end{aligned}
		\]
		with this last step using that $\frac{-1}{s},\frac{-1}{st-1}>\frac{-1}{s-1}$ and that the numerator of the middle term of the minimum can be rewritten as, $(2s-3)(t-1)+s-1$, which is non-negative (and hence trivially is more than $\frac{-1}{s-1}$).  Thus condition (d) holds and we are able to apply \Cref{lemma:BSviaSHADOW}.

		Let $C$ be as guaranteed by \Cref{lemma:BSviaSHADOW} and let $C_r\ge C$ be sufficiently large so that
		$$
		M_r(n):=\max\l\{M_3(n)^{\frac{r-1}{2}}n^{-\frac{r-3}{2}},~n^{r-1}\r\}(\log n)^{C}\le n^{\alpha_r}(\log n)^{C_r},
		$$
		$$
		\gamma_r(n):=\gamma_3(n)(\log n)^C\le (\log n)^{C_r},
		$$
		and
		$$
		\begin{aligned}
			\tau_r(n,m)&:=\tau_3\l(n,~n^{\frac{r-3}{r-1}}m^{\frac{2}{r-1}}(\log n)^{-C}\r)(\log n)^C.\\
			&\le \max\l\{\l(n^{\frac{r-3}{r-1}}m^{\frac{2}{r-1}}\r)^{-1/3}n^{1-1/s},~\l(n^{\frac{r-3}{r-1}}m^{\frac{2}{r-1}}\r)^{\frac{-t}{st+3t-2}}n^{\frac{3st-3t-s+2}{s(st+3t-2)}},~n^{\frac{-(s-1)}{s(st-1)}}\r\}^{s-1} n^{2-\frac{1}{s}}(\log n)^{C_r}\\
			&=\max\l\{m^{-\frac{2}{3(r-1)}}n^{\frac{2r}{3(r-1)}-\frac{1}{s}},~m^{\frac{-2t}{(r-1)(st+3t-2)}}n^{\frac{2rt}{(r-1)(st+3t-2)}-\frac{3t+s-2}{s(st+3t-2)}},~n^{\frac{-(s-1)}{s(st-1)}}\r\}^{s-1} n^{2-\frac{1}{s}}(\log n)^{C_r}.
		\end{aligned}
		$$
		
		Then by \Cref{lemma:BSviaSHADOW} we know that $K^{(r)}_{s,t}$ is $(M_r,~\gamma_r,~\tau_r)$-balanced, which completes the proof.
		
	\end{proof}
	
	Now we are ready to prove \Cref{thm:Kst}, which we restate below for convenience. 
	
	\begingroup
	\def\thethm{\ref{thm:Kst}}
	\begin{thm}
		If $r,s,t\ge 3$ are integers such that $t\ge s$ and either
		\begin{itemize}
			\item $r\ge 2s/3+2$, or
			\item $r\ge \lfloor 2s/3+2\rfloor$ and $t\ge 8s/3$,
		\end{itemize}  
		then for all $0<p=p(n)\le 1$ we have a.a.s.
		$$
		\ex(G^r_{n,p}, K_{s,t}^{(r)})=
		\l\{
		\begin{aligned}
			& \max\l\{pn^{r-1},n^{2-\frac{s+t-2}{st-1}}\r\}(\log n)^{\Theta(1)},~~&\t{if}~p\gg n^{-r+2-\frac{s+t-2}{st-1}}\\
			&(1+o(1))p\binom{n}{r},~~&\t{if}~ n^{-r}\ll p\ll  n^{-r+2-\frac{s+t-2}{st-1}}.\\
		\end{aligned}
		\r.
		$$
	\end{thm}
	\addtocounter{thm}{-1}
	\endgroup
	
	\begin{proof}
		Note that the $r$-density of $K^{(r)}_{s,t}$ is
		$$
		d_r(K^{(r)}_{s,t}):=\max_{F'\subseteq K^{(r)}_{s,t},\ |F'|\ge 2}\l\{\frac{|F'|-1}{v(F')-r}\r\}=\frac{|K^{(r)}_{s,t}|-1}{v(K^{(r)}_{s,t})-r}=\frac{st-1}{(r-2)st+s+t-r},
		$$
		and hence
		\[\frac{1}{d_r(K_{s,t}^{(r)})}=\frac{(r-2)st+s+t-r}{st-1}=r-2+\frac{s+t-2}{st-1}.\]
		By \Cref{lem:generalLowerRandomTuran} we conclude all of the bounds of \Cref{thm:Kst} except for
		\iffalse 
		$$
		\ex(G^r_{n,p}, K_{s,t}^{(r)})
		\l\{
		\begin{aligned}
			& \ge \max\l\{pn^{r-1},n^{2-\frac{s+t-2}{st-1}}\r\}(\log n)^{\Theta(1)},~~&\t{if}~p\gg n^{-r+2-\frac{s+t-2}{st-1}},\\
			&= (1+o(1))p\binom{n}{r},~~&\t{if}~ n^{-r}\ll p\ll  n^{-r+2-\frac{s+t-2}{st-1}}.\\
		\end{aligned}
		\r.
		$$\fi 
		$$\ex(G^r_{n,p}, K_{s,t}^{(r)})\le  \max\l\{pn^{r-1},n^{2-\frac{s+t-2}{st-1}}\r\}(\log n)^{\Theta(1)}$$ when $p\gg n^{-r+2-\frac{s+t-2}{st-1}}$. 
		\iffalse 
		Let $p_0:=n^{-r+3-\frac{s+t-2}{st-1}}$, noting that this is exactly the point at which the two terms in $\max\l\{pn^{r-1},n^{2-\frac{s+t-2}{st-1}}\r\}$ equal each other.  We claim that to complete the proof, it suffices to prove $\ex(G^r_{n,p}, K_{s,t}^{(r)})\le  pn^{r-1}(\log n)^{\Theta(1)}$ for $p\ge p_0$.  Indeed, if this holds, then by monotonicity we would have for all $p\le p_0$ that 
		\[\ex(G^{r}_{n,p}, K^{(r)}_{s,t})\le \ex(G^{r}_{n,p_0}, K^{(r)}_{s,t})\le p_0 n^{r-1}= n^{2-\frac{s+t-2}{st-1}}(\log n)^{\Theta(1)},\]
		finishing the proof.  As such, from now on we only consider $p\ge p_0$ and aim to prove $\ex(G^r_{n,p}, K_{s,t}^{(r)})\le  pn^{r-1}(\log n)^{\Theta(1)}$. \fi Let $C_r$ be the constant guaranteed by \Cref{thm:supersaturation-r}. We will apply \Cref{Lemma:General Random Turan} with $F=K^{(r)}_{s,t}$, $M(n)=n^{r-1}(\log n)^{C_r}$, and 
		$$
		\tau_r(n,m)=\max\l\{m^{-\frac{2}{3(r-1)}}n^{\frac{2r}{3(r-1)}-\frac{1}{s}},~m^{\frac{-2t}{(r-1)(st+3t-2)}}n^{\frac{2rt}{(r-1)(st+3t-2)}-\frac{3t+s-2}{s(st+3t-2)}},~n^{\frac{-(s-1)}{s(st-1)}}\r\}^{s-1} n^{2-\frac{1}{s}}(\log n)^{C_r}.
		$$
		By \Cref{thm:supersaturation-r} and the definition of $\tau_r$, one can check that both conditions (a) and (b) in \Cref{Lemma:General Random Turan} hold, and we have $pM(n)\to \infty$ as $n\to \infty$ for $p\gg n^{-r+2-\frac{s+t-2}{st-1}}$. Thus, by \Cref{Lemma:General Random Turan} applied with $m=M(n)$, we have for $C'$ and $C''$ sufficiently large constants that\footnote{To verify this, it can help to realize that each term in the maximum of $\tau_r$ is of the form $m^{-x}n^{rx-y}$, so plugging in $m=n^{r-1}$ reduces these terms to $n^{x-y}$, after which we multiply these exponents by $(s-1)$ and add $1+(s-1)/s$ to them, ultimately leaving these terms in the form $n^{(s-1)x+1+\frac{(s-1)}{s}[1+sy]}$.} 
		$$
		\begin{aligned}
			\ex(G^r_{n,p},K^{(r)}_{s,t})&\le \max\l\{C'pM(n),\tau_r(n,M(n))(\log n)^{C'}\r\} \\
			&\le \max\l\{pn^{r-1},~n^{\frac{2(s-1)}{3(r-1)}+1},~n^{\frac{2t(s-1)}{(r-1)(st+3t-2)}+1+\frac{(s-1)(t-1)}{st+3t-2}},~n^{1+\frac{(s-1)(t-1)}{st-1}}\r\} (\log n)^{C''}.
		\end{aligned}
		$$
		
		To complete the proof, it suffices to show that each of the middle two terms in the maximum above is upper bounded by the last term, which by subtracting 1 from each exponent and dividing by $s-1$ is equivalent to showing
		\[\frac{2}{3(r-1)}\le \frac{t-1}{st-1}\hspace{3em}\textrm{and}\hspace{3em} \frac{2t}{(r-1)(st+3t-2)}+\frac{t-1}{st+3t-2}\le \frac{t-1}{st-1}.\]
		Each of these two inequalities are equivalent to a lower bound on $r$ which in total combine into
		
		\iffalse 
		$$
		\max\l\{\frac{2}{3(r-1)},~\frac{2t}{(r-1)(st+3t-2)}+\frac{t-1}{st+3t-2}\r\}\le \frac{t-1}{st-1}.
		$$
		Each term in the maximum 
		
		which by multiplying each side by $\frac{(r-1)(st-1)}{t-1}$ and then adding 1 is equivalent to showing \SS{The second term here isn't so clear...at all; ah I see each individual inequality above gives some bound on $r$ and this is what they are.}\fi 
		\begin{equation}\label{eq:rBound}
			r\ge \max\l\{1+\frac{2(st-1)}{3(t-1)},~1+\frac{2(st-1)t}{(3t-1)(t-1)}\r\}=1+\frac{2(st-1)t}{(3t-1)(t-1)}.
		\end{equation}
		
		All that remains then is to show that the $r$ as in the hypothesis of \Cref{thm:Kst} satisfy \eqref{eq:rBound}.  First consider the case that $r\ge \frac{2s}{3}+2$, in which case using $3\le s\le t$ implies
		
		$$
		\begin{aligned}
			\frac{2s}{3}+2-\l(1+\frac{2(st-1)t}{(3t-1)(t-1)}\r)&=\frac{(2s+3)(3t-1)(t-1)-6(st-1)t}{3(3t-1)(t-1)}\\
			&=\frac{9t^2-6t+3-2s(4t-1)}{3(3t-1)(t-1)}\\
			&\ge \frac{9t^2-6t+3-2t(4t-1)}{3(3t-1)(t-1)}\\
			&=\frac{t-3}{3(3t-1)}\\
			&\ge 0,
		\end{aligned}
		$$
		verifying \eqref{eq:rBound} in this case.  Now assume $r\ge \frac{2s}{3}+\frac{4}{3}$ (which in particular implies $r\ge \lfloor \frac{2s}{3}+2\rfloor$), in which case $s\le 3t/8$ implies
		
		$$
		\begin{aligned}
			\frac{2s}{3}+\frac{4}{3}-\l(1+\frac{2(st-1)t}{(3t-1)(t-1)}\r)&=\frac{3t^2+2t+1-2s(4t-1)}{3(3t-1)(t-1)}\\
			&\ge \frac{3t^2+2t+1-\frac{6}{8}t(4t-1)}{3(3t-1)(t-1)}\\
			&=\frac{\frac{11}{4}t+1}{3(3t-1)(t-1)}\\
			&\ge 0.
		\end{aligned}
		$$
		
		We conclude that \eqref{eq:rBound} always holds for our potential values of $r$, finishing the proof.
		
	\end{proof}
	As an aside, our proof above implicitly shows that the bounds of \Cref{thm:Kst} hold whenever \eqref{eq:rBound} holds, which can be used to get slightly better bounds on $t$ in \Cref{thm:Kst}.  For example, when $s\equiv 2 \mod 3$ we can improve the lower bound on $t$ to roughly $4s/3$, and in the $s\equiv 1\mod 3$ case we can improve the bound on $t$ by 1.
	
	\iffalse
	Note that this bound on $t$ is pretty close to best possible asymptotically, and even at $s=4$ it gives $t\ge 11$ (the true limit is $t\ge 10$, which is quite close).  One can do slightly better by having $t\ge 8s/3-11/32$ (or just $8s/3-1/3$ for simplicity) but this probably isn't worth it since it really just changes the bound by 1 (and it genuinely will whenever $s\equiv 1\mod 3$ I think).
	It definitely isn't worth including, but when $s\equiv 2\mod 3$ one can get a bound closer to $t\ge 4s/3$.
	\fi 
	
	\section{Proof of \Cref{prop:sparse}}\label{subsec:Sparse}
	
	Here we prove our second graph supersaturation result \Cref{prop:sparse}, which we recall was the only ingredient that we needed in our proof of \Cref{prop:technicalSupersaturation} whose proof we deferred.  We begin by recalling our notation and restating (an equivalent version of) this result.
	
	Given a tripartite graph $G$ together with a set of parameters, we defined a function $\phi:2^{V(G)}\to \R\cup\{\infty\}$ as follows: given $\nu\sub V(G)$, let $a=|\nu\cap (V_1\cup V_2)|$ and $b=|\nu\cap V_3|$.  If $a=0$ or $b=0$, then $\phi(\nu)=\infty$.  Otherwise, if $\nu\cap V_1\ne \emptyset$ we define 
	\[\phi(\nu)=\frac{\ell^{s-1} n^{s-1+1/s}(\ell^{3-2s}K^{s-1})^t}{\del m_{1,3}(\del \ell^{-1} K n^{1-2/s})^{a-1}(\del \ell^{3-2s} K^{s-1})^{b-1}}\]
	and if $\nu\cap V_1=\emptyset $ we define
	\[\phi(\nu)=\frac{\ell^{s-1} n^{s-1+1/s}(\ell^{3-2s}K^{s-1})^t}{\del m_{2,3}(\del \ell^{-1} K n^{1-2/s})^{a-1}(\del \ell^{3-2s} K^{s-1})^{b-1}}.\]

	Recall that \Cref{prop:sparse} essentially says that if $G$ is a tripartite graph satisfying certain conditions, then we can (in particular) find a large collection $\c{G}$ of $K_{s,t}$'s in $G$ which has no edges in $V_1\cup V_2$, where we view $\c{G}$ as an $st$-uniform hypergraph whose vertex set is $E(G)$.  Observe crucially that knowing the edge set of a $K_{s,t}$ which has no edges in $V_1\cup V_2$ is equivalent to knowing the vertex set of this $K_{s,t}$.  
	
	Because of this, instead of working with the $st$-uniform hypergraph $\c{G}$ above, it will be slightly more convenient to work with the (equivalent) $(s+t)$-uniform hypergraph $\c{G}_v$ which has vertex set $V(G)$ whose hyperedges encode copies of $K_{s,t}$ in $G$ which have no edges in $V_1\cup V_2$.  With this in mind, \Cref{prop:sparseV} can be written in the following equivalent form (noting that the only change in this statement comes from working with $\c{G}_v$ instead of $\c{G}$).

	\begin{prop}\label{prop:sparseV}
		For all $s,t\ge 3$, there exists a sufficiently small $\del>0$  such that the following holds.  Let $\ell,K,n$ be real numbers and let $H$ be a tripartite 3-graph with tripartition $V_1\cup V_2\cup V_3$ with $m_{i,j}$  the number of pairs in $V_i\cup V_j$ contained in at least one hyperedge.  Further, assume $\ell,K,n, H$ satisfy the following:
		\begin{itemize}
			\item[(a)] $|V_i|\le n$ for all $i$.
			\item[(b)] $H$ has at least $K n^{3-3/s}(12\log n)^{16}$ hyperedges.
			\item[(c)] $m_{i,j}\le \ell |V_1|^{2-1/s}$ for all $i,j$.
			\item[(d)] We have $\del^{-1}\le \ell\le \del K^{(s-1)/(2s-3)}$. 
			\item[(e)] Each pair in $V_i\cup V_j$ is contained in at most $K n^{3-3/s} m_{i,j}^{-1} (12\log n)^{32}$ hyperedges for all $i,j$.
			\item[(f)] Each vertex in $V_1$ has at most $\ell |V_1|^{1-1/s} (12\log n)^{16}$ neighbors in $V_2$.
		\end{itemize}
		Then for $G=\partial^2 H$, there exists a collection $\c{G}_v$ of $K_{s,t}$'s in $G$ (viewed as a hypergraph on $V(G)$) such that every $K_{s,t}$ of $\c{G}_v$ intersects $V_1$ in 1 vertex and $V_2$ in $s-1$ vertices, such that $|\c{G}_v|=\Om^*(\ell^{s-1} n^{s-1+1/s}(\ell^{3-2s}K^{s-1})^t)$, and such that for any $\nu\sub V(G)$ we have $\deg_{\c{G}_v}(\nu)\le \phi(\nu)$.
	\end{prop}
	It is not difficult to see that \Cref{prop:sparseV} is equivalent to \Cref{prop:sparse}, so it remains to prove this result.  We also note for later that (d) implies that $\ell$ and $K$ (and hence $n$) are sufficiently large in terms of $s,t$. 
	
	\begin{proof}
		
		We begin by establishing a series of preliminary claims that will be of use to us later on.
		
		\begin{claim}\label{cl:floor}
			We have $\phi(\nu)\ge 1$ for all $\nu\sub V(G)$ with $|\nu\cap (V_1\cup V_2)|\le s$ and $|\nu \cap V_3|\le t$.  In particular, $\floor{\phi(\nu)}\ge \half \phi(\nu)$ for all such $\nu$ with $\phi(\nu)<\infty$.
		\end{claim}
		
		\begin{proof}
			Let $a=|\nu\cap (V_1\cup V_2)|$ and $b=|\nu\cap V_2|$. 
			If $a=0$ or $b=0$ then $\phi(\nu)=\infty$ and there is nothing to show, so we can assume $a,b\ge 1$.   Observe that $\phi$ is decreasing in $a,b$ (since $\ell\le \del K^{(s-1)/(2s-3)}\le \del K$ implies $\del \ell^{-1} K n^{1-2/s}\ge 1$ and $\del \ell^{3-2s}K^{s-1}\ge 1$).  Thus, for all such $\nu$, we have for some $i\in \{1,2\}$ that 
			\[\phi(\nu)\ge \frac{\ell^{s-1} n^{s-1+1/s}\cdot (\ell^{3-2s}K^{s-1})^t}{\del m_{i,3}(\del \ell^{-1} K n^{1-2/s})^{s-1}(\del \ell^{3-2s} K^{s-1})^{t-1}}= \del^{-s-t+1} m_{i,3}^{-1} \ell n^{2-1/s}\ge\del^{-s-t+1},\]
			where this last step used (c) and (a) to get $m_{i,3}\le \ell |V_1|^{2-1/s}\le \ell n^{2-1/s}$.  The result follows for $\del$ sufficiently small.
		\end{proof}

		For some slight ease of notation, we define 
		\[M:=\ell^{s-1} n^{s-1+1/s}(\ell^{3-2s}K^{s-1})^t,\]
		noting that this is the numerator in the definition of $\phi(\nu)$ and is the lower bound we wish to prove for $|\c{G}_v|$ up to logarithmic factors. We will say that a collection $\c{G}_v$ of $K_{s,t}$'s in $G$ is \textit{$\phi$-bounded} if $\deg_{\c{G}_v}(\nu)\le \phi(\nu)$ for all $\nu\sub V(G)$, and if every $K_{s,t}$ in $\c{G}_v$ has its part of size $s$ intersecting $V_1$ in 1 vertex and $V_2$ in $s-1$ vertices.  Let $\c{G}_v$ be the largest $\phi$-bounded collection of $K_{s,t}$'s.  If $|\c{G}_v|\ge M(12 \log n)^{-16}$ then we are done, so from now on we may assume for contradiction that 
		\begin{equation}
			|\c{G}_v|<M(12 \log n)^{-16} \label{eq:smallGv}
		\end{equation}
		
		Our goal now is to show that there exist more than $|\c{G}_v|$ copies of $K_{s,t}$ in $G$ such that $\c{G}_v$ together with any of these copies continues to be $\phi$-bounded, contradicting the maximality of $\c{G}_v$.  Motivated by this, we define the set of \textit{saturated sets}
		\[\c{F}=\{\nu\sub V(G): \deg_{\c{G}_v}(\nu)\ge \floor{\phi(\nu)}\},\]
		and we observe that $\c{G}_v$ together with an additional copy $\c{K}$ of $K_{s,t}$ which intersects $V_1$ in 1 vertex and $V_2$ in $s-1$ vertices is $\phi$-bounded if and only if $\nu\notin \c{F}$ for all $\nu\sub  \c{K}$.  M
		Motivated by this, we say that a set $\c{K}\sub V(G)$ is \textit{$\c{F}$-good} if $\nu\notin \c{F}$ for any $\nu\sub \c{K}$.  In particular, we note that our desired contradiction will be achieved if we can find strictly more than $|\c{G}_v|$ $\c{F}$-good sets $\c{K}$ which induce $K_{s,t}$'s in $G$ as prescribed by the proposition (since in this case, $\c{G}_v$ together with one of these copies will be a larger $\phi$-bounded collection), and from now on this shall be our goal.
		
		For a copy to be $\c{F}$-good, it must in particular avoid any edges $uv$ with $\{u,v\}\in \c{F}$.  As such, we must restrict ourselves to the following subgraph.
		\begin{claim}\label{cl:H'}
			Let $G'\sub G$ consist of all edges $uv$ with $\{u,v\}\notin \c{F}$. Then $G'$ contains at least $\half K n^{3-3/s}(12\log n)^{16}$ triangles.
		\end{claim}
		As an aside, this will be the only part of the proof where we use that $G$ is the shadow of a hypergraph.
		\begin{proof}
			For integers $i,j\in \{1,2,3\}$, we say that an edge $\nu=\{v_i,v_j\}$ in $G$ is $(i,j)$-bad if $\nu\in \c{F}$ and $v_i\in V_i,\ v_j\in V_j$.  As a subclaim, we aim to show that the number of $(i,j)$-bad pairs is at most $\rec{6}m_{i,j}(\log n)^{-16}$ for any $i,j$.
			
			To show this, fix some $i,j$ and let $\nu$ be an $(i,j)$-bad pair.  If $3\notin \{i,j\}$ then $\phi(\nu)=\infty$, so no such pair exists and there is nothing to prove.  Thus without loss of generality we can assume $j=3$.  By definition of $\c{F}$, \Cref{cl:floor}, and the definition of $\phi$; having $\nu\in \c{F}$ means \[\deg_{\c{G}_v}(\nu)\ge \floor{\phi(\nu)}\ge \half \phi(\nu)=\half \del^{-1} m_{i,3}^{-1}M.\]With this, a  double counting argument gives
			\begin{align*}\#(i,3)\textrm{-bad pairs}&\le \# \left[\textrm{edges } \nu \textrm{ with }\deg_{\c{G}_v}(\nu)\ge \half \del^{-1} m_{i,3}^{-1}M \right]\\ &\le \left(\half \del^{-1}m_{i,3}^{-1}M \right)^{-1}\cdot \sum_{\nu:|\nu|=2} \deg_{\c{G}_v}(\nu) = 2\del m_{i,3} M^{-1} \cdot {s+t\choose 2} |\c{G}_v|\\ &< 2\del m_{i,3} M^{-1} \cdot {s+t\choose 2} M(12 \log n)^{-16}\le \rec{6} m_{i,3} (12 \log n)^{-16},\end{align*}
			with the second to last inequality using \eqref{eq:smallGv} and the last inequality using that $\del$ is sufficiently small, proving the sub-claim.
			
			Now define $H'\sub H$ to consist of all hyperedges which do not contain pairs $\{u,v\}\in \c{F}$, and observe that the number of triangles in $G'$ is at least $|H'|$, so it suffices to lower bound $|H'|$.  Also note by (e) that each $(i,j)$-bad edge is contained in at most $Kn^{3-3/s} m_{i,j}^{-1}(12\log n)^{32}$ hyperedges of $H$, so by the subclaim, the number of hyperedges removed going from $H$ to $H'$ is at most
			\[\sum_{i,j} \rec{6}m_{i,j}(12 \log n)^{-16}\cdot Kn^{3-3/s} m_{i,j}^{-1}(12\log n)^{32}=  \half K n^{3-3/s}(12 \log n)^{16}.\]
			Thus by (b), \[|H'|\ge |H|-\half K n^{3-3/s}(12\log n)^{16}\ge \half K n^{3-3/s}(12\log n)^{16},\] proving the result.
		\end{proof}
		
		By working with $G'$, we will automatically guarantee that every edge of the $K_{s,t}$'s contained in $G'$ lie outside of $\c{F}$.  To avoid larger subsets of $\c{F}$, we define for $\nu\sub V(G)$ the \textit{link set} 
		\[\c{J}(\nu)=\{u\in V(G): \nu\cup \{u\}\in \c{F}\}.\]
		The idea here is that if we have already selected some  $\c{F}$-good set $\c{K}'$ that we aim to include in our ultimate copy $\c{K}$ of $K_{s,t}$, then $\c{K}'\cup \{u\}$ will continue to be $\c{F}$-good if and only if $u\notin \c{J}(\nu)$ for any $\nu\sub \c{K}'$.  Thus $\bigcup_{\nu\sub \c{K}'} \c{J}(\nu)$ represents the set of ``bad choices'' for $u$ given the partially constructed copy $\c{K}'$. 
		The following claim gives an upper bound on the number of these ``bad choices'' for $u$ which lie within some given set of vertices $V$.

		\begin{claim}\label{cl:link}
			Let $\nu,V\sub V(G)$.  If $\phi(\nu\cup \{u\})=\infty$ for all $u\in V$, then $\c{J}(\nu)\cap V=\emptyset$, and otherwise
			\[|\c{J}(\nu)\cap V|\le 2(s+t) \frac{\phi(\nu)}{\min_{u\in V} \phi(\nu\cup \{u\})}.\]
		\end{claim}
		\begin{proof}
			
			If $\phi(\nu\cup \{u\})=\infty$, then trivially $\deg(\nu\cup \{u\})<\phi(\nu\cup \{u\})$, which means $\nu\cup \{u\}\notin \c{F}$ and hence $u\notin \c{J}(\nu)$.  This gives the result if $\phi(\nu\cup \{u\})=\infty$ for all $u\in V$, so from now on we assume $\min_{u\in V} \phi(\nu\cup \{u\})<\infty$.
			
			Observe that
			\[\sum_{u\in \c{J}(\nu)\cap V} \deg_{\c{G}_v}(\nu\cup \{u\})\le \sum_{u\in V(G)} \deg_{\c{G}_v}(\nu\cup \{u\})= (s+t) \deg_{\c{G}_v}(\nu)\le (s+t) \phi(\nu),\]
			where the equality used that each hyperedge $h$ containing $\nu$ is counted exactly $(s+t)$ times in the second summation (namely, once  for each $u\in h$), and the last inequality used that $\c{G}_v$ is $\phi$-bounded.  On the other hand, the definition of $\c{J}(\nu)$ yields
			\[\sum_{u\in \c{J}(\nu)\cap V} \deg_{\c{G}_u}(\nu\cup \{u\})\ge \sum_{u\in \c{J}(\nu)\cap V} \floor{\phi(\nu\cup \{u\})}\ge |\c{J}(\nu)\cap V|\cdot \min_{u\in V} \floor{\phi(\nu\cup \{u\})}.\]
			Rearranging these two inlined expressions gives 
			\[|\c{J}(\nu)\cap V|\le (s+t) \frac{\phi(\nu)}{\min_{u\in V} \floor{\phi(\nu\cup \{u\})}},\]
			and the final result follows by using \Cref{cl:floor} to get rid of the floor.
		\end{proof}
		
		With these preliminaries established, we begin our main argument for finding many copies $\c{K}$ which are $\c{F}$-good through a modification of the usual ``star counting'' argument of K\H{o}v\'ari-S\'os-Tur\'an.
		\begin{claim}\label{cl:P}
			Let $\c{P}$ denote the set of pairs $(S,v)$ such that the following holds:
			\begin{itemize}
				\item $v\in V_3$,
				\item $S\sub N(v)$ with $|S\cap V_1|=1$ and $|S\cap V_2|=s-1$,
				\item The vertex $u_1\in S\cap V_1$ is adjacent to every vertex of $S\cap V_2$, and
				\item $S\cup \{v\}$ is $\c{F}$-good.
			\end{itemize}
			
			Then \[|\c{P}|\ge (8s)^{1-s} \ell n^{2-1/s}\left(\ell^{-1}Kn^{1-2/s}(12\log n)^{16}\right)^{s-1}.\]
		\end{claim}
		\begin{proof}
			Given an edge $uv$ in $G'$, let $\deg(u,v)$ denote the number of common neighbors of $u,v$ in $G'$, i.e.\ the number of triangles containing the edge $uv$.  Let $E$ be the set of edges $uv$ in $V_1\cup V_3$, and let $E'\sub E$ be those edges with
			\begin{equation}\deg(u,v)\ge \ell^{-1}Kn^{1-2/s} \ge 4s,\label{eq:PCodegree}\end{equation}
			with this second inequality holding since $\ell\le \del K$ and $\del$ is sufficiently small.  Note that
			\begin{equation}\sum_{uv\in E'} \deg(u,v)\ge \frac{1}{2} Kn^{3-3/s}(12\log n)^{16}- \ell^{-1}Kn^{1-2/s}\cdot |E|\ge \frac{1}{4} Kn^{3-3/s}(12\log n)^{16},\label{eq:PTriangles}\end{equation}
			where the first inequality used that the total number of triangles in $G'$ is at least $\half Kn^{3-3/s}(12\log n)^{16}$  by \Cref{cl:H'} and that the number of triangles not containing an edge of $E'$ is at most $\ell^{-1}Kn^{1-2/s}\cdot m_{1,3}$ by definition of $E'$; and the second inequality used (c) and (a) to get $|E|\le m_{1,3}\le \ell |V_1|^{2-1/s}\le  \ell n^{2-1/s}$ and that $n$ is sufficiently large.
			
			We will now build elements of $\c{P}$ iteratively as follows:  we start with any edge $u_1v\in E'$ with $u_1\in V_1$ (noting that $\{v,u_1\}$ is $\c{F}$-good by definition of it being an edge in $G'$), and given that we have already selected $v,u_1,\ldots,u_{i-1}$ with $i\le s$ such that $\{v,u_1,\ldots,u_{i-1}\}$ is $\c{F}$-good, we next choose $u_i$ to be any vertex in $N(u_1)\cap N(v)\sm \{u_2,\ldots,u_{i-1}\}$ such that $\{v,u_1,\ldots,u_i\}$ is $\c{F}$-good.
			
			Note that if $(S,v)$ with $S=\{u_1,\ldots,u_s\}$ is the output of this procedure, then $(S,v)$ is an element of $\c{P}$ by construction, so it remains to argue that there are many ways to complete this procedure.  Observe that given $u_1\in V_1,v\in V_3$, the number of choices for $u_i$ with $i>1$ in this algorithm is at least
			\[\deg(u_1,v)-(i-1)-\sum_{\nu\sub \{v,u_1,\ldots,u_{i-1}\}}|\c{J}(\nu)\cap (V_2\cap N(v))|,\]
			where the $i-1$ term denotes the number of choices $u_i=u_j$ for some $j<i$, and the summation is an overestimate for the number of choices $u_i\in V_2\cup N(v)$ which causes $\{v,u_1,\ldots,u_i\}$ to not be $\c{F}$-good.

			We aim to show that each term in the sum above is at most $2(s+t) \del \ell^{-1} Kn^{1-2/s}$.  Indeed, if $\nu$ contains $v$ and at least one vertex of $V_1\cup V_2$, then $\phi(\nu)=\del \ell^{-1} Kn^{1-2/s}\cdot \phi(\nu\cup \{u_i\})\ne \infty$ for all $u_i\in V_2$, so the result follows from \Cref{cl:link}.  Similarly if $v\notin \nu$ then $\phi(\nu\cup \{u_i\})=\infty$ for all $u_i\in V_2$ (since $\nu\cap V_3=\emptyset$), so \Cref{cl:link} gives that this set is empty.  The only remaining case is $\nu=\{v\}$.  By definition of $G'$, any $u_i\in N(v)$ is such that $\{u_i,v\}$ is $\c{F}$-good, so $\c{J}(\nu)\cap N(v)=\emptyset$ in this case, again giving the result.
			
			In total, we conclude that for $i>1$ the number of choices for $u_i$ in the algorithm is always at least
			\[    \deg(u_1,v)-s-2^{s+t}\cdot 2(s+t) \del \ell^{-1} Kn^{1-2/s}\ge \half \deg(u_1,v)\]
			with the last step holding for $\del$ sufficiently small by \eqref{eq:PCodegree}.  This in turn implies that the total number of ways of proceeding through the algorithm is at least
			\[
			\begin{aligned}
				\sum_{uv\in E'} \left(\half \deg(u,v)\right)^{s-1}&\ge |E'|\left(\half |E'|^{-1}\sum_{uv\in E'} \deg(u,v)\right)^{s-1}\\
				&\ge   \ell n^{2-1/s}\left(\half\cdot \ell^{-1}n^{-2+1/s}\cdot \frac{1}{4}Kn^{3-3/s}(12\log n)^{16}\right)^{s-1},
			\end{aligned}
			\]
			where the first inequality used convexity and the second used \eqref{eq:PTriangles} and $|E'|\le m_{1,3}\le \ell n^{2-1/s}$.  This gives the desired result after noting that number of \textit{distinct} $(S,v)$ pairs generated by the algorithm is at least the quantity above (i.e.\ the number of ways of terminating with the algorithm) divided by $s!\le s^{s-1}$ (since this is an upper bound for the number of ways a given pair $(S,v)$ can be generated by this algorithm).
		\end{proof}
		A similar argument gives the following.
		\begin{claim}\label{cl:S}
			There exist $\Om(\ell^{s-1} n^{s-1+1/s} (\ell^{3-2s} K^{s-1})^t(\log n)^{16(s-1)})$ sets $\c{K}=\{u_1,\ldots,u_s,v_1,\ldots,v_t\}$ forming a $K_{s,t}$ in $G'$ such that $u_1\in V_1$, $u_i\in V_2$ for $i>1$, $v_j\in V_3$ for all $j$, and such that $\c{K}$ is $\c{F}$-good.
		\end{claim}
		As a point of comparison for later, we note that the number of these $\c{K}$ that we guarantee is substantially more than
		\[M=\ell^{s-1} n^{s-1+1/s}(\ell^{3-2s}K^{s-1})^t,\]
		provided $n$ is sufficiently large.
		\begin{proof}
			Given a set $S$, define $\c{P}(S)=\{v:(S,v)\in \c{P}\}$ and let $\deg(S)=|\c{P}(S)|$.  Define $\c{S}$ to be the collection of sets $S$ with $(S,v)\in \c{P}$ for some $v$, and let $\c{S}'$ be the collection of sets $S$ with 
			\[
			\deg(S)\ge (8s)^{-s}\ell n^{2-1/s}|V_1|^{-1}(\ell^{-2}K n^{1-2/s}|V_1|^{-1+1/s})^{s-1},
			\]
			noting that $|V_1|\le n$ and $\ell\le \del K^{(s-1)/(2s-3)}$ implies
			\begin{equation}\deg(S)\ge (8s)^{-s}\ell^{3-2s}K^{s-1} \ge 4t\textrm{ for }S\in\c{S}'.\label{eq:SCodegree}\end{equation}
			Also note that \begin{equation}\label{eq:SSize}
				|\c{S}'|\le |\c{S}|\le |V_1| (\ell |V_1|^{1-1/s}(12\log n)^{16})^{s-1},
			\end{equation}
			since every $S\in \c{S}$ includes a vertex of $V_1$ together with $s-1$ of its neighbors in $V_2$ (and the number of these neighbors is bounded by (f)); and that
			\begin{align}
				\sum_{S\in \c{S}'} \deg(S)&\ge |\c{P}|-(8s)^{-s}\ell n^{2-1/s}|V_1|^{-1}(\ell^{-2}K n^{1-2/s}|V_1|^{-1+1/s})^{s-1} \cdot |\c{S}|\nonumber \\
				&\ge |\c{P}|-(8s)^{-s}\ell n^{2-1/s}|V_1|^{-1}(\ell^{-2}K n^{1-2/s}|V_1|^{-1+1/s})^{s-1} \cdot |V_1| (\ell |V_1|^{1-1/s}(12\log n)^{16})^{s-1}\nonumber\\ 
				&= |\c{P}|-(8s)^{-s}\ell n^{2-1/s}(\ell^{-1} K n^{1-2/s}(12\log n)^{16})^{s-1} \ge \half |\c{P}| \label{eq:STriangle},
			\end{align}    
			where the first inequality used that $\sum_{S\in \c{S}} \deg(S)=|\c{P}|$ and that the number of pairs missed by summing over $\c{S}'$ is at most the quantity we subtract by, the second inequality used \eqref{eq:SSize}, and the final inequality used the lower bound on $|\c{P}|$ from \Cref{cl:P} (which has a $1-s$ exponent for $8s$ instead of a $-s$ exponent).
			
			We will iteratively form our sets $\c{K}$ as follows: start with any set $S\in \c{S}'$.  Iteratively given some $v_1,\ldots,v_{i-1}\in V_3$, choose $v_i\in V_3\sm\{v_1,\ldots,v_{i-1}\}$ to be such that $(S,v_i)\in \c{P}$ and such that $S\cup \{v_1,\ldots,v_{i}\}$ is $\c{F}$-good, and finally output the set $\c{K}=S\cup \{v_1,\ldots,v_t\}$.  It is not difficult to see that each $\c{K}$ outputted by this procedure is of the desired form, so it remains to show that there are many ways of going through the procedure.  
			
			Given some $S\in \c{S}'$, it is not difficult to see that the number of choices for $v_i$ is at least
			\[\deg(S)-(i-1)-\sum_{\nu\sub S\cup \{v_1,\ldots,v_{i-1}\}}|\c{J}(\nu)\cap (V_3\cap \c{P}(S))|.\]
			We aim to show that each term in this sum is at most $2(s+t)\del \ell^{3-2s}K^{s-1}$.  Indeed, if $\nu$ intersects both $V_1\cup V_2$ and $V_3$ then $\phi(\nu)=\del \ell^{3-2s}K^{s-1} \phi(\nu\cup \{v_i\})\ne \infty$ for all $v_i\in V_3$, so the result follows from \Cref{cl:link}.  If $\nu$ is disjoint from $V_1\cup V_2$ then $\phi(\nu\cup \{v_i\})=\infty$ for all $v_i\in V_3$, so \Cref{cl:link} gives that this set is empty.  If $\nu$ is disjoint from $V_3$ and $v_i\in \c{P}(S)$, then $\nu\cup \{v_i\}\sub S\cup \{v_i\}$ with $(S,v_i)\in \c{P}$, so by definition of $\c{P}$ we have that $\nu\cup \{v_i\}$ is $\c{F}$-good, so again we conclude that this set is empty, giving the result.
			
			In total, we conclude that the number of choices for $v_i$ in the algorithm is always at least 
			\[\deg(S)-t-2^{s+t}\cdot 2(s+t)\del \ell^{3-2s}K^{s-1} \ge \half \deg(S), \]
			with the last step holding for $\del$ sufficiently small by \eqref{eq:SCodegree}.  This in turn implies that the total number of ways of proceeding through the algorithm is at least 
			
			\begin{align*}\sum_{S\in \c{S}'} \left(\half \deg(S)\right)^t&\ge |\c{S}'|\left(\half |\c{S}'|^{-1}\sum_{S\in \c{S}'}\deg(S)\right)^t\ge |\c{S}'|\left(\frac{1}{4} |\c{S}'|^{-1}|\c{P}|\right)^t\\
				&\ge \Om\left(|V_1|(\ell |V_1|^{1-1/s}(12\log n)^{16})^{s-1}\cdot (\ell n^{2-1/s}|V_1|^{-1}(\ell^{-2}K n^{1-2/s}|V_1|^{-1+1/s})^{s-1})^t \right),\end{align*}
			where the first inequality used convexity, the second used \eqref{eq:STriangle}, and the last used \eqref{eq:SSize} and \Cref{cl:P}.  Using $|V_1|\le n$ (together with the fact that the exponent of $|V_1|$ in the expression above is negative), we in total get the desired lower bound for the number of ways to complete the procedure, and since each $\c{K}$ arises in at most $O(1)$ ways from this procedure, we conclude the desired result.
		\end{proof}
		Observe that the number of $\c{K}$ from \Cref{cl:S} is substantially more than $|\c{G}_v|<M(12 \log n)^{-16}$ provided $\ell$ (and hence $K$ and hence $n$) is sufficiently large. In particular, there exists some $\c{K}$ as in \Cref{cl:S} which is not in $\c{G}_v$.  By construction, $\c{G}_v\cup \{\c{K}\}$ is a larger $\phi$-bounded collection of $K_{s,t}$'s, contradicting our choice of $\c{G}_v$.  We conclude the result.
	\end{proof}
	\section{Concluding Remarks}
	In this paper we proved a balanced supersaturation result for $K_{s,t}^{(3)}$, yielding improved ranges on the $r$ for which we can give tight bounds for the random Tur\'an number of expansions $K_{s,t}^{(r)}$.  In particular, our methods yield an optimal dependency on $r$ for $s=4$, and it is natural to try and extend this to other values of $s$.  As discussed just after \Cref{cor:4}, the easiest $s$ to solve this problem for are those of the form $1+{k\choose 2}$, motivating the following.
	\begin{prob}
		Establish tight bounds for $\ex(G_{n,p}^r,K_{7,t}^{(r)})$ whenever $r\ge 5$ and $t$ is sufficiently large.
	\end{prob}
	Analogous to the methods used in this paper, the most natural approach for solving this problem would be to prove a balanced supersaturation result for $K_{s,t}^{(4)}$ directly, after which one can use our lifting theorems to give strong bounds for all $r\ge 5$.
	
	While the above is a natural direction to pursue, there are significant complications in going from uniformity 3 to uniformity 4 with our methods.  Notably, our present proof for 3-uniform hypergraphs largely boiled down to using two distinct arguments based on whether our 3-graph $H$ has a large or small number of edges in the 2-shadow $\partial^2 H$ relative to the number of edges in $H=\partial^3 H$.  For uniformity 4, one seemingly needs to consider now how the relative sizes of all the shadows $\partial^k H$ for $2\le k\le 4$ compare to each other, giving a much more complex set of possible parameters that one would have to build arguments around.  %To be clear, we have not been able to rule out the impossibility of such an approach, only that it seems to be significantly more complex to implement than our current arguments here.

	In \Cref{prop:vanillaSupersaturation} we recorded a (non-balanced) supersaturation result for $K_{s,t}^{(3)}$.  We claim without proof that one can easily lift this supersaturation result to all $r$-expansions of $K_{s,t}$ using the methods from \cite{nie2024random} to yield the following.
	
	\begin{prop}\label{prop:vanillaRUniform}
		For all $t\ge s\ge r$, there exists a sufficiently large constant $C$ such that if $H$ is an $n$-vertex $r$-graph with $m$ edges such that 
		$$m\ge\max\{n^{r-1},~n^{r-\frac{r-1}{2}(3/s-1/st)}\}(\log n)^{C},$$ 
		then $H$ contains at least
		$$
		\Omega^*(m^{st}n^{s+t-2st})
		$$
		copies of $K^{(r)}_{s,t}$.
	\end{prop}
	The count of $\Omega^*(m^{st}n^{s+t-2st})$ here is tight whenever it applies.  Moreover, the bound on $m$ in \Cref{prop:vanillaRUniform} is best possible whenever $n^{r-1}$ achieves the maximum, which happens whenever $r\ge \frac{2st}{3t-1}+1\approx 2s/3$.  In particular, it is plausible that improvements to the bound of $m$ in \Cref{prop:vanillaRUniform} for smaller values of $r$ could lead towards improving our $r\approx 2s/3$ bound for random Tur\'an numbers of $K_{s,t}^{(r)}$ in \Cref{thm:Kst}. 
	
	With an eye towards optimizing \Cref{prop:vanillaRUniform} for smaller $r$, it is plausible that the true bound for $m$ for any value of $r$ should be
	\[m=\Om^*(\max\{n^{r-1},n^{r-{r\choose 2}/s}\}),\]
	as this is the Tur\'an number of $K_{s,t}^{(r)}$ whenever $t$ is sufficiently large in terms of $s$.  Our inability to close this gap on $m$ from \Cref{prop:vanillaRUniform} even for the case $K_{s,t}^{(3)}$ is closely related to the gap in the corresponding graph supersaturation problem stated in \cite{dubroff2023clique}, which asks to determine the number of copies of $K_{s,t}$ guaranteed to exist in a graph with a given number of triangles.
	
	We strongly believe that our methods used in this paper for $K_{s,t}^{(3)}$ can be used to make additional progress on this corresponding graph supersaturation problem, but decided not to complicate this paper any further by including this.   We note, however, that one can not use our main technical result \Cref{prop:sparse} as a blackbox towards this problem, as we required the graph used in this proposition to be the shadow of a nice hypergraph.  As we noted in its proof, this shadow hypothesis is only needed to prove \Cref{cl:H'}, which itself is only needed to prove a \textit{balanced} supersaturation result, so this in and of itself is not a significant obstacle to the problem.  There are a few other additional complications that arise when trying to shift our approach to the graph setting, though again we do not believe any of them should seriously impede our methods going through in this setting.

	\bibliographystyle{abbrv}
	\bibliography{refs}

\end{document}